\documentclass[11pt,twoside,a4paper]{article}
\usepackage[T1]{fontenc}
\usepackage{lmodern}
\usepackage[applemac]{inputenc}
\usepackage[english]{babel}
\usepackage{epsfig}
\usepackage{mathrsfs}
\usepackage{color}
\usepackage{amsthm}
\usepackage{amsmath}
\usepackage{amsfonts}
\usepackage{amssymb}
\usepackage{graphicx}
\usepackage{graphics}
\usepackage{subfigure}

\usepackage{pdfsync}

\usepackage{fullpage}

\usepackage{MesCommandes}

\title{A remark on the single scattering preconditioner applied to boundary integral equations\thanks{The author is partially supported by the French ANR
fundings under the project MicroWave NT09\_460489.}}
\author{B. Thierry \\
Institut Mont\'efiore, Universit\'e de Li\`ege, \\
ACE Team, Sart-Tilman Bldg. B28, \\
B-4000 Li\`ege, Belgium. \\
{\tt B.Thierry@ulg.ac.be}}
\date{}

\begin{document}

\maketitle

\renewcommand{\thefootnote}{\fnsymbol{footnote}}

%\footnotetext[2]{Institut Mont\'efiore,
   % Universit\'e de Li\`ege, ACE Team, Sart-Tilman Bldg. B28, B-4000 Li\`ege, Belgique 
%({\tt B.Thierry@ulg.ac.be}).}

\renewcommand{\thefootnote}{\arabic{footnote}}

\begin{abstract}
This article deals with boundary integral equation preconditioning for the multiple scattering problem. The focus is put on the single scattering preconditioner, corresponding to the diagonal part of the integral operator, for which two results are proved. Indeed, after applying this geometric preconditioner, it appears that, firstly, every direct integral equations become identical to each other, and secondly, that the indirect integral equation of Brakhage-Werner becomes equal to the direct integral equations, up to a change of basis. These properties imply in particular that the convergence rate of a Krylov subspaces solver will be exactly the same for every preconditioned integral equations. To illustrate this, some numerical simulations are provided at the end of the paper.
\end{abstract}

%\tableofcontents
%=======================================================================
\section{Introduction}%\label{secEqInt:EqInt}
%=======================================================================

This paper is devoted to the numerical resolution of acoustic multiple scattering problems in the time-harmonic regime. Here, multiple scattering means that the medium contains more than one obstacle contrary to single scattering where only one scatterer is considered. Multiple scattering arises in many applications and problems such as wave propagation in photonics crystals or trabecular bones modeling. Among all the numerical methods to solve scattering problems, the focus is put here on boundary integral equations. These approaches, based on the integral representation of the scattered field, reduce the initial boundary value problem to an integral equation on the surface of the obstacles, decreasing the dimension by one. However, after discretization, the non-localness of the integral operator leads to a full matrix. Moreover, the mesh of the domain must be sufficiently fine to capture the oscillatory behavior of the wave and hence, at high frequency and/or with a large number of obstacles, the size of the linear system becomes very large. Hence and due to their computational cost, direct solvers cannot be employed and the numerical solution is then handled by a Krylov subspace solver, such as GMRES \cite{SaaSch86}, which can be improved via at least two major changes. First, the CPU cost of the matrix-vector product involved at each iteration can be reduced from $\GrandO{n^{2}}$ to $\GrandO{n\ln(n)}$ thanks to Fast Multipole Method (FMM, see \textit{e.g.} \cite{CoiRokWan93, Dar00, Dar00b}). Secondly, the system can be preconditioned to enhance the convergence rate of the Krylov solver, which is greatly slowed down by the non-positiveness of the Helmholtz operator, especially at high frequency. Finding a robust and efficient preconditioner is non-trivial and is more complicated when the FMM is employed since the matrix of the system is not computed. In particular, only a few number of algebraic preconditioners can be applied, such as SPAI (SParse Approximate Inverse, see \textit{e.g.} \cite{GroHuc97, CarDufGir05}). Another possibility is to work with a well-conditioned integral equation and numerous analytic preconditioning techniques has already been proposed, as for example in \cite{AloBorLev07, AntoineDarbasQJMAM, AntoineDarbasM2AN, DarDarLaf13}.
In multiple scattering context, a natural preconditioner is the one representing single scattering effects. For $M$ obstacles and given the matrix of a discretized integral equation, this preconditioner is composed by the $M$ blocks located on the diagonal of this matrix. As every block represent the scattering problem by one obstacle, this geometric preconditioner is called single scattering preconditioner. 
%Because each of these block represents the scattering by only one obstacle, this geometric preconditioner is called single scattering preconditioner.

%For $M$ obstacles and given an integral formulation, this preconditioner is obtained by solving each of the $M$ single scattering problem and concatenating the $M$ matrices in a block diagonal matrix. As this geometric preconditioner involves the solution of single scattering problems, it can be employed in addition to previously mentioned preconditining techniques. 

This article focuses on the effects of this preconditioning on boundary integral equations. They are first studied on direct boundary integral equations, for which the unknown densities are exactly the Cauchy data. After being preconditioned by their single scattering preconditioner, it appears that the three classical direct integral equations become exactly the same! This surprising result is moreover independent of the geometry. Secondly, this preconditioner is applied to the indirect integral equation of Brakhage-Werner, where the unknown densities do not have a physical meaning anymore, and a similar result is obtained. More precisely, it turns out that the preconditioned Brakhage-Werner integral equation 
is similar\footnote{In this paper, two bounded operators $S$ and $T$ acting on a Hilbert space X are called similar if there exists a invertibly bounded operator U on X such that $S = U^{-1}TU$.} to the preconditioned direct integral equations.
%is equal to the preconditioned direct integral equations, up to a change of basis.
As a consequence, every integral equations will share the same spectral properties after being composed by their single scattering preconditioner, and the iterative solver will have the exact same convergence rate. These results are the central point of this article and, as far as the author knows, are new.

This paper begins by a general presentation of boundary integral formulations and a systematic way of building direct integral equations. This part is well-known but mandatory to prove the main results of this article. The first result, namely the equality between the preconditioned direct integral equation, is established in section \ref{sec:SingleScat}. The case of the indirect integral equation of Brakhage-Werner is studied thereafter in section \ref{sec:BW}. Then in section \ref{sec:num}, a numerical example using $\Pb^{1}$ finite elements discretization shows that the matrices of the discretized boundary integral equations share the same spectrum, and thus confirms the theoretical results. Finally, the paper ends with a short conclusion.

%In this article, the focus is put on this single scattering preconditioner. It is first applied to direct boundary integral equations (also named field integral equations in electromagnetics), for which a necessary recall on their construction is given in section \ref{secEqInt:EqInt}. Surprisingly, it appears that, after being preconditioned, the three classical direct integral equations become exactly the same! This result, shown in section \ref{sec:SingleScat}, is moreover independent of the geometry. As far as the author knows, this surprising result is new. Secondly, this result is extended in section \ref{sec:BW} for the indirect integral equation of Brakhage-Werner. More precisely, in that case, the equality holds up to an invertible operator. As a consequence,  every integral equations will share the same spectral properties after being composed by their single scattering preconditioner, and the iterative solver will have the exact same convergence rate, as illustrated in section \ref{sec:num} by numerical results.

%
%=======================================================================
\section{Classical direct integral equations}\label{secEqInt:EqInt}
%=======================================================================

%We recall here a method to obtain the classical boundary integral equation (BIE), that is taken from the lecture of Bendali and Fares \cite{BenFar07} and also the PhD thesis of Darbas \cite{Dar04}.

Here is presented a systematic method to obtain the usual boundary integral equation EFIE, MFIE and CFIE where the acronyms EFIE, MFIE and CFIE design here respectively \emph{Electric Field Integral Equation}, \emph{Magnetic Field Integral Equation} and \emph{Combined Field Integral Equation}. This section is widely inspired by the lecture of Bendali and Fares \cite{BenFar07} but is mandatory to introduce notations. More details can also be found for example in the the PhD thesis of the author \cite{Thi11} or of Darbas \cite{Dar04}.

	\subsection{The scattering problem}

Let the whole domain $\Rb^d$ be filled with a homogeneous and nondissipative medium, where $d=2,3$ is the dimension of the problem. Let also $\Omegam$ be a  bounded open set of $\Rb^{d}$ such that the propagation domain $\Omegaps =\Rb^{d}\setminus\overline{\Omegam}$ is connected. The boundary $\Gamma$ of $\Omegam$ is assumed to be smooth enough (say of class $C^{2}$) with a unit normal vector $\nn$ directed into $\Omegaps$. The illumination of the scatterer $\Omegam$ by a time-harmonic incident wave $\uinc$ gives rise to a scattered field $u$, solution of the following scattering problem (the time dependence is assumed to be of the form $e^{-i \omega t}$ and the wavenumber $k$ real and positive):
\begin{equation}\label{eqEqInt:ProblemeU}
\begin{cases}
        \Delta u + k^2 u = 0 &\text{in }\Omega^+, \\
        u = -\uinc & \text{on } \Gamma, \\
		u \text{ outgoing,}
\end{cases}
\end{equation}
where $\Delta = \sum_{j=1}^d\frac{\partial^{2}}{\partial x_{j}^{2}}$ is the Laplace operator and the outgoing condition stands for the Sommerfeld radiation condition:
$$
\displaystyle{\lim_{\|\xx\| \to \infty} \|\xx\|^{(d-1)/2}\left( \nabla u  \cdot \frac{\xx}{\|\xx\|} - iku\right)=0,}
$$
with $\|\xx\| = (\sum_{j=1}^{d}x_{j}^{2})^{1/2}$. % denotes the euclidian norm on $\Rb^d$. 
This paper is restricted to a Dirichlet boundary condition (sound-soft obstacles) and an incident plane wave of direction $\Beta$:
$$
\forall \xx\in\Rb^{d}, \qquad \uinc(\xx) = e^{ik\Beta\cdot\xx}.
$$
However, the results presented in this article remain true with either a Neumann boundary condition and/or any incident wave smooth enough in a neighborhood of the obstacle $\Omegam$ (\textit{e.g.} a time harmonic wave emitted by point source, \textit{e.g.} Green function).
To obtain the direct integral equations, it is more convenient to work with the total field $\ut = u +\uinc$, solution of the following problem ($\ut$ exists in $\Omegaps$ since $\uinc$ is plane and thus is a solution of the Helmholtz equation in $\Rb^{d}$)
\begin{equation}\label{eqEqInt:ProblemeUt}
\begin{cases}
        \Delta \ut + k^2 \ut = 0 &\text{in }\Omega^+, \\
        \ut = 0 & \text{on } \Gamma, \\
		(\ut-\uinc) \text{ outgoing.}
\end{cases}
\end{equation}
Recall that both problems \eqref{eqEqInt:ProblemeU} and \eqref{eqEqInt:ProblemeUt} are uniquely solvable \cite{ColKre83}:
\begin{theorem}\label{theo:UniqueSolution}
The scattering problems (\ref{eqEqInt:ProblemeU}) and (\ref{eqEqInt:ProblemeUt}) admit a unique solution.
\end{theorem}

%====================================================================
	\subsection{Integral operators}\label{secEqInt:OpInt}
%====================================================================

Let the volume single-layer integral operator $\Lop$ be defined by (see \textit{e.g.} \cite[Theorem 6.12]{McL00})
$$
\begin{array}{c c c l l}
\Lop: & \hmdemi & \longrightarrow & H^{1}_{loc}(\Rb^d) &\\
& \rho & \longmapsto & \Lop \rho, & \dsp{  \forall \xx \in \Rb^d, \quad \Lop\rho(\xx) = \int_\Gamma G(\xx,\yy) \rho(\yy) \, \dd\Gamma(\yy)},
% \Lop \rho(\xx) = \PSdemi{\rho}{G(\xx,\cdot)}},
\end{array}
$$
and the volume double-layer integral operator $\Mop$ by
$$
\begin{array}{c c c l l }
\Mop: & \hdemi & \longrightarrow & H^{1}_{loc}(\Rb^d \setminus \Gamma) & \\
& \lambda & \longmapsto & \Mop \lambda, & \dsp{\forall \xx \in \Rb^d \setminus \Gamma,  \Mop\lambda(\xx) = -\int_\Gamma \dny G(\xx,\yy) \lambda(\yy) \, \dd\Gamma(\yy)},
%\Mop \lambda(\xx) = \PSdemi{\dn G(\xx,\cdot)}{\lambda}}.
\end{array}
$$
where the spaces $\hmdemi$, $H^{1}(\Rb^d \setminus \Gamma)$, $H^1_{loc}(\Rb^d \setminus \Gamma)$ are the usual Sobolev spaces and the Green function $G(\cdot\,,\cdot)$ is given by
\begin{equation}\label{eqEqInt:Green}
\forall \xx, \yy \in \Rb^d, \;\xx \neq \yy, \qquad G(\xx,\yy) = 
\begin{cases}
\dsp{\frac{i}{4}H_0^{(1)}(k\|\xx-\yy\|)}, & \text{ if } d = 2,\\[0.2cm]
\dsp{\frac{e^{ik\|\xx - \yy\|}}{4\pi\|\xx - \yy\|}}, & \text{ if } d = 3.
\end{cases}
\end{equation}
\begin{remark}
All the integrals on $\Gamma$ must be seen as a dual product between the Sobolev space $H^{1/2}(\Gamma)$ and its dual $H^{-1/2}(\Gamma)$. However, as soon as the data ($\uinc$ and $\Gamma$) are smooth enough, then the scattered field $u$ is also smooth and the dual product can be identified with the (non-hermitian) scalar product on $L^2(\Gamma)$:
$$
 \PSdemi{f}{g} = \int_\Gamma f(\xx)g(\xx)\;\dd\Gamma(\xx).
$$
This identification is considered throughout this paper.
\end{remark}

The trace $\gammazpm$ and the normal trace $\gammaupm$ operators are now defined following and inspired by \cite[Appendix A]{ChaGraLan12}, where the plus or minus sign specifies whether the trace is taken from the inside of $\Omegaps$ or $\Omegam$. First, the trace operators $\gammazpm: H^1(\Omegapm) \to H^{1/2}(\Gamma)$ are defined so that, if $v\in C^{\infty}(\overline{\Omegapm})$, then
$$
\gammazpm v(\xx) = \lim_{\zz \in \Omegapm \to \xx}v(\zz),
$$
for almost every $\xx\in\Gamma$. By introducing the space $H^1(\Omegapm; \Delta) := \{v\in H^1(\Omegapm) ; \Delta v\in L^2(\Omegapm)\}$ and the linear operators $\gammazpmd:H^{1/2}(\Gamma)\to H^1(\Omegapm)$ such that $\gammazpm\gammazpmd\varphi=\varphi$, for all $\varphi \in H^{1/2}(\Gamma)$, the normal traces $\gammaupm: H^1(\Omegapm; \Delta)\to H^{-1/2}(\Gamma)$ can be defined \cite[Equation (A.28)]{ChaGraLan12}:
\begin{multline}\label{eq:gammaupm}
\forall v\in H^{1}(\Omegapm; \Delta), \forall \varphi\in H^{1/2}(\Gamma),\\
\left(\gammaupm v, \varphi\right)_{H^{-1/2}(\Gamma), H^{1/2}(\Gamma)} 
:= \mp\left[\int_{\Omegapm}\Delta v(\xx)\overline{w(\xx)}\;\dd\xx + \int_{\Omegapm}\nabla v(\xx)\cdot\nabla \overline{w(\xx)}\;\dd\xx\right],
\end{multline}
where  $w:=\gammazpmd \varphi$ (and thus satisfies $\gammazpm w = \varphi$). As the quantities involved in scattering problem do not belong to $H^1(\Omegaps)$ but to $H^1_{loc}(\Omegaps)$, the exterior trace and normal trace operators are naturally extended as $\gammazps:H^1_{loc}(\Omegaps)\to H^{1/2}(\Gamma)$ and $\gammaups:H^1_{loc}(\Omegaps;\Delta)\to H^{-1/2}(\Gamma)$ by $\gammazps(v)=\gammazps(vv')$ and $\gammaups(v)=\gammaups(vv')$, where $v'$ is an arbitrary compactly supported and indefinitely differentiable function on $\overline{\Omegaps}$ which is equal to $1$ in a neighborhood of $\Gamma$, and where $H^1_{loc}(\Omegaps; \Delta) := \{v\in H^1_{loc}(\Omegaps) ; \Delta v\in L^2_{loc}(\Omegaps)\}$. Remark that, when the function $v$ is sufficiently smooth, then its normal trace $\gammaupm v$, given by (\ref{eq:gammaupm}), belongs to $L^2(\Gamma)$ and can be written as $\gammaupm v(\xx) = \lim_{\zz\in\Omegapm\to\xx}\nabla v(\zz)\cdot \nn(\xx)$, for almost every $\xx$ on $\Gamma$.
Note also that, the single- and double-layer potentials, introduced previously, belong not only to $H^1_{loc}(\Omegaps)\bigcup H^1(\Omegam)$ but also to $H^1_{loc}(\Omegaps; \Delta)\bigcup H^1(\Omegam; \Delta)$ (see \eg \cite[\S2.2]{ChaGraLan12}).
Some well-known properties of the single- and double-layer potentials are summarized in the following propositions. Their proof can be found for example in \cite[Theorems 7.5 and 9.6]{McL00} for proposition \ref{propEqInt:potentiel} and in \cite[Theorem 6.12]{McL00} for proposition \ref{propEqInt:trace}.
\begin{prop}\label{propEqInt:potentiel}
For every densities $\rho \in \hmdemi$ and $\lambda \in \hdemi$, the single-layer potential $\Lop\rho$ and double-layer potential $\Mop\lambda$ are outgoing solutions of the Helmholtz equation in $\Rb^d \setminus \Gamma$. Moreover, the scattered field $u$, solution of (\ref{eqEqInt:ProblemeU}), can be written as
$$
\forall\xx\in\Omegaps,\qquad u(\xx) = -\Lop(\dn u|_{\Gamma}) (\xx)-\Mop(u|_{\Gamma}) (\xx).
$$
\end{prop}
%Let us first define the trace $\gammazpm$ and the normal trace $\gammaupm$ where the plus or minus sign specifies whether the trace is taken from the inside of $\Omegaps$ or $\Omegam$: %For any point $\xx\in\Gamma$, 
%$$
%\forall\xx\in\Gamma,\qquad \gammazpm g(\xx) := \lim_{\zz\in \Omega^{\pm} \to \xx} g(\zz)
%\qquad\text{ and }\qquad
%\gammaupm g(\xx) := \lim_{\zz\in \Omega^{\pm} \to \xx} \dnz g(\zz).
%$$
\begin{prop}\label{propEqInt:trace}
The trace and the normal trace of the operators $\Lop$ and $\Mop$ are given by the following relations %(the sign denote if the point $z$ tends to $x$ from the inside or the outside of $\Gamma$)
%For every point $\xx$ of $\Gamma$, the trace and the normal trace of the operators $\Lop$ and $\Mop$ are given by the following relations (the sign denote if the point $z$ tends to $x$ from the inside or the outside of $\Gamma$) 
\begin{equation}\label{eqEqInt:trace}
\begin{array}{l @{\qquad \qquad}l }
\dsp{\gammazpm\Lop \rho =L \rho}, &
\dsp{\gammazpm \Mop\lambda = \left(\mp \frac{1}{2}I + M\right) \lambda},\\[0.3cm]
\dsp{\gammaupm \Lop \rho  = \left( \mp \frac{1}{2}I + N\right)\rho},&
\dsp{\gammaupm\Mop\lambda = D \lambda},
%\dsp{\lim_{\zz\in \Omega^{\pm} \to \xx} \Lop \rho (\zz) =L \rho (\xx)}, &
%\dsp{\lim_{\zz\in \Omega^{\pm} \to \xx} \Mop\lambda(\zz) = \left(\mp \frac{1}{2}I + M\right) \lambda(\xx)},\\[0.3cm]
%\dsp{\lim_{\zz\in \Omega^{\pm} \to \xx} \dnz \Lop \rho (\zz) = \left( \mp \frac{1}{2}I + N\right)\rho(\xx)},&
%\dsp{\lim_{\zz\in \Omega^{\pm} \to \xx} \dnz \Mop\lambda(\zz) = D \lambda(\xx)},
\end{array}
\end{equation}
where $I$ is the identity operator and, for $\xx\in\Gamma, \rho \in\hmdemi$ and $\lambda\in\hdemi$, the four boundary integral operators are defined by
\begin{equation}\label{eqEqInt:OpIntBord2}
\begin{array}{l l c l @{\quad\qquad}l c l}
L : & H^{-1/2}(\Gamma) & \longrightarrow & \dsp{H^{1/2}(\Gamma),} & \dsp{L\rho(\xx)} &  =  &\dsp{ \int_{\Gamma} G(\xx,\yy) \rho(\yy) \dd\Gamma(\yy)}, \\[0.3cm]
M : & H^{1/2}(\Gamma) & \longrightarrow & \dsp{H^{1/2}(\Gamma),} & \dsp{M\lambda(\xx) } &  =  &\dsp{  -\int_{\Gamma} \dny G(\xx,\yy)\lambda(\yy) \dd\Gamma(\yy)}, \\[0.3cm]
N : & H^{-1/2}(\Gamma) & \longrightarrow & \dsp{H^{-1/2}(\Gamma),} & \dsp{N\rho(\xx)  }&  =  &\dsp{  \int_{\Gamma} \dnx G(\xx,\yy) \rho(\yy) \dd\Gamma(\yy) = -M^* \rho (\xx)}, \\[0.3cm]
D : & H^{1/2}(\Gamma) & \longrightarrow & \dsp{H^{-1/2}(\Gamma),} &\dsp{D\lambda (\xx) }  &  =  &\dsp{  -\dnx \int_{\Gamma} \dny G(\xx,\yy)\lambda (\yy) \dd\Gamma(\yy)}.
\end{array}
\end{equation}
\end{prop}
In this paper, the boundary integral operators are written with a roman letter ({\it e.g.} $L$) whereas the volume integral operators are written with a calligraphic letter ({\it e.g.} $\Lop$). 
%\begin{remark}
%The boundary integral operator $-N$ is the adjoint of the operator $M$, in the sense that
%$$
%\forall (f,g)\in H^{1/2}(\Gamma)\times H^{-1/2}(\Gamma), \qquad\qquad\PSdemi{g}{Mf} = \PSdemi{-Ng}{f}.
%$$
%
%\end{remark}
%
%According to \cite[Theorems 4.4.1]{Ned01}, the operators $M$ and $N$ are compact (they are regularizers of order 1, $i.e.$ continuous from $H^{s}(\Gamma)$ into $H^{s+1}(\Gamma)$).
%\begin{prop}\label{propEqInt:MNcompacts}
%%Les opérateurs $L : H^{-1/2}(\Gamma) \longrightarrow  \dsp{H^{1/2}(\Gamma)}$ et $D : H^{1/2}(\Gamma) \longrightarrow  \dsp{H^{1/2}(\Gamma)}$ définissent des isomorphismes. 
%If the boundary $\Gamma$ is of class $C^2$ then the operator $M$ (respectively the operator $N$) is compact from $H^{1/2}(\Gamma)$ into $H^{1/2}(\Gamma)$ (respectively from $H^{-1/2}(\Gamma)$ into $H^{-1/2}(\Gamma)$). 
%\end{prop}
%On the other side, the boundary integral operators $L$ and $D$ are invertible, providing $k$ is not an irregular frequency (see $e.g.$ theorems 3.4.1 and 3.4.2 from \cite{Ned01}).
According to \cite[Theorems 3.4.1 and 3.4.2]{Ned01}, the boundary integral operators $L$ and $D$ are invertible, providing $k$ is not an irregular frequency.
 \begin{theorem}\label{theo:FreqIr}
Let $F_D(\Omegam)$ (resp. $F_N(\Omegam)$) be the countable set of positive wavenumbers $k$ accumulating at infinity such that the interior homogeneous Dirichlet (resp. Neumann) problem
\begin{equation}\label{eqEqInt:problemeinterne1}
\begin{cases}
-\Delta v = k^2 v & \text{in }\Omega^-, \\
v = 0 \left(\text{resp. } \dn v = 0 \right)& \text{on } \Gamma, \\ 
%\gammazm v = 0 \left(\text{resp. } \gammaum v = 0 \right)& \text{on } \Gamma, \\ 
\end{cases}
\end{equation}
admits non-trivial solutions. Then, the operator $L$ (resp. $D$) realizes an isomorphism from $\hmdemi$ into $\hdemi$ (resp. from $\hdemi$ into $\hmdemi$) if and only if $k \not\in F_D(\Omegam)$ (resp. $k \not\in F_N(\Omegam)$).
\end{theorem}
These irregular frequencies $k$ of $F_D(\Omegam)$ (resp. of $F_N(\Omegam)$) are exactly the square roots of the eigenvalues of the Laplacian operator $(-\Delta)$ for the homogeneous interior Dirichlet (resp. Neumann) problem. In the multiple scattering case, that is when $\Omegam = \bigcup_{p=1}^M \Omegamp$ is multiply connected, the following equalities clearly hold true
\begin{equation}\label{eq:FDOmegamp}
F_D(\Omegam) = \bigcup_{p=1}^M F_D(\Omegamp) \qquad\text{ and } \qquad F_N(\Omegam) = \bigcup_{p=1}^M F_N(\Omegamp).
\end{equation}
Throughout the paper, $F_{DN}(\Omegam)$ denotes the set of all irregular frequencies:
\begin{equation}\label{eq:FDN}
F_{DN}(\Omegam) = F_{D}(\Omegam)\bigcup F_{N}(\Omegam).
\end{equation}

%----------------------------------------------------------------------------------------------------
%\subsection{Direct integral equations}
%----------------------------------------------------------------------------------------------------

%We present here a systematic method to obtain the usual boundary integral equation EFIE, MFIE and CFIE where the acronyms EFIE, MFIE and CFIE design here respectively \emph{Electric Field Integral Equation}, \emph{Magnetic Field Integral Equation} and \emph{Combined Field Integral Equation}. More details can be found for example in the lecture of Bendali and Fares \cite{BenFar07}, the PhD of the author \cite{Thi11} or of Darbas \cite{Dar04}.

\subsection{Direct integral equations}

This section details the way of deriving direct integral equations, described in \cite{BenFar07}. This approach is nonstandard but has advantages that appear later in the paper at Section \ref{sec:SingleScat}.

The principle is to write the total field $\ut$ as a combination of a single- and a double-layer potentials:
\begin{equation}\label{eqEqInt:ut}
\ut(\xx) = \Lop \rho (\xx) + \Mop \lambda (\xx) + \uinc(\xx), \qquad \forall \xx \in \Omegaps,
\end{equation}
where $(\lambda, \rho)$ are now the two unknown of the problem. Thanks to proposition \ref{propEqInt:potentiel}, such an expression ensures that both $\ut$ is solution of the Helmholtz equation in $\Omega^+$ and $(\ut-\uinc)$ is outgoing. % (see e.g \cite{ColKre83}). 
Following \cite{BenFar07}, an integral equation is said to be direct when the densities $(\lambda,\rho)$ have a physical meaning. Indeed, for these integral equations, they are exactly the Cauchy data $\left( -\ut|_\Gamma, -\dn\ut|_\Gamma \right)$. However, this is not a choice but a consequence of the construction of the integral equation. In electromagnetic scattering, direct and indirect integral equations are more often referred to as respectively \emph{field} and \emph{source} integral equations (see \textit{e.g.} Harington and Mautz \cite{HarMau78, MauHar79} or PhD thesis of Sophie Borel \cite{Bor06}).

For now on, the problem, composed by the two unknown $(\lambda,\rho)$, has only one equation given by the Dirichlet boundary condition on $\Gamma$. To obtain a second equation, a fictitious interior wave $\utm$, living in $\Omegam$, is introduced and defined by
\begin{equation}\label{eqEqInt:utm}
\utm(\xx) = \Lop \rho (\xx) + \Mop \lambda (\xx) + \uinc(\xx), \qquad \forall \xx \in \Omegam.
\end{equation}
Remark that, on the one hand $\utm$ is a solution of the Helmholtz equation in $\Omegam$ and on the other hand, due to the trace relations (\ref{eqEqInt:trace}), the couple of unknown $(\lambda,\rho)$ satisfies the well-known \emph{jump-relation}
\begin{equation}\label{eqEqInt:jump}
\left\{
\begin{array}{l}
\lambda = \utm|_\Gamma - \ut|_\Gamma, \\[0.2cm]
\rho = \dn \utm|_\Gamma - \dn \ut|_\Gamma.
\end{array}
\right.
\end{equation}
As the wave $\utm$ is fictitious, it does not act on the solution $\ut$ of the scattering problem. As a consequence, the boundary condition on $\Gamma$ imposed to $\utm$ has no influence on $\ut$. Let this constraint be represented by an operator $A$ such that $\utm$ is the solution of the following interior problem
\begin{equation}\label{eqEqInt:ProblemeInterne}
\begin{cases}
        \Delta \utm + k^2 \utm = 0 &\text{in }\Omegam, \\
        A \utm = 0 & \text{on } \Gamma.
	\end{cases}
\end{equation}
To build a direct integral equation, the operator $A$ is chosen such that the field $\utm$ vanishes in $\Omegam$.% (this method is also named ``null-field method''). 
Supposing that such an operator exists, then, on the boundary $\Gamma$, the following equalities will hold true
$$
\begin{cases}
\utm|_\Gamma = 0,  & \\
\dn\utm|_\Gamma = 0. &
\end{cases}
$$
Consequently and thanks to the Dirichlet boundary condition $\ut|_{\Gamma}=0$,  the jump relations (\ref{eqEqInt:jump}) will read as
$$
\begin{cases}
\lambda = 0,&\\%-\ut|_\Gamma, & \\
\rho = -\dn\ut|_\Gamma, &
\end{cases}
$$
%and the set of unknowns $(\lambda,\rho)$ will exactly be the set of Cauchy data. Moreover, the Dirichlet boundary condition $\ut|_{\Gamma} = 0$ of the initial scattering problem (\ref{eqEqInt:ProblemeUt}) will imply that
%$$
%\begin{cases}
%\lambda = 0, & \\
%\rho = -\dn\ut|_\Gamma. &
%\end{cases}
%$$
Therefore, both the fictitious field $\utm$ and the total field $\ut$ will be composed by a single-layer potential only
$$
\left\{\begin{array}{l}
\dsp{\ut(\xx) = \Lop \rho (\xx) + \uinc(\xx), \qquad \forall\xx \in \Omegaps,}\\[0.2cm]
\dsp{\utm(\xx) = \Lop \rho (\xx) + \uinc(\xx), \qquad \forall\xx \in \Omegam.}
\end{array}\right.
$$
The unknown $\rho$ is finally obtained through the resolution of the (direct) integral equation $A\utm = 0$, which can be written as
\begin{equation}\label{eqEqInt:EqInt}
A\Lop\rho = -A\uinc.
\end{equation}
%with $\LA = A\Lop$.
%Finally, the solution $\ut$ of the scattering problem (\ref{eqEqInt:problemeinterne1}) is built as a single-layer potential with the density $\rho$, solution of (\ref{eqEqInt:EqInt}):
%$$
%\ut(\xx) = \Lop \rho (\xx) + \uinc(\xx), \qquad \forall\xx \in \Omegaps.
%$$ 
Both the expression and the nature of the integral equation (\ref{eqEqInt:EqInt}) depend on the boundary condition imposed to $\utm$, represented here by the operator $A$.  The next subsections describe the three usual direct integral equations that are studied in this paper. The proofs are not provided and can be found for example in \cite{BenFar07} or \cite{Thi11}.

%%%%%%%%%%%%%
\subsubsection{EFIE (\textit{Electric Field Integral Equation})}
%%%%%%%%%%%%%

%The EFIE is obtained when a homogeneous Dirichlet boundary condition is imposed on the fictitious field $\utm$. 
For this integral equation, the operator $A$ is the interior trace operator $\gammazm$ on $\Gamma$. Thanks to the continuity on $\Gamma$ of the single-layer integral operator $\Lop$ (see equation (\ref{eqEqInt:trace})), the boundary integral equation (\ref{eqEqInt:EqInt}) becomes
\begin{equation}\label{eq:EFIE}
L \rho = - \uincg.
\end{equation}
Due to theorem \ref{theo:FreqIr}, this first kind integral equation, named \emph{Electric Field Integral Equation} (EFIE), is well-posed and equivalent to the scattering problem (\ref{eqEqInt:ProblemeUt}) 
except for Dirichlet irregular frequencies. %, for which the field $\utm$ is not necessarily null. 
%This is summarized in the following Proposition.
\begin{prop}\label{prop:EFIE}
If $k\not\in F_D(\Omegam)$ then the single-layer potential $\Lop\rho + \uinc$ is solution of the scattering problem (\ref{eqEqInt:ProblemeUt}) \ssi $\rho$ is the solution of the EFIE (\ref{eq:EFIE}). %Moreover, in that case, the density $\rho$ is equal to $-\dn\ut|_{\Gamma}$.
\end{prop}

\begin{remark}
When $k\in F_{D}(\Omegam)$, the integral operator $L$ is no more bijective but is still one-to-one. It can be shown that the kernel of the operator $L$ is a subset of the kernel of the operator $\Lop$. Consequently, for every solution $\rhot$ of the EFIE, the associated single-layer potential $\Lop\rhot + \uinc$ is still the solution of the scattering problem (\ref{eqEqInt:ProblemeUt}).
\end{remark}

%%%%%%%%%%%%%%
\subsubsection{MFIE (\textit{Magnetic Field Integral Equation})}

Another possibility is to chose $A = \gammaum$, the interior normal trace. Using traces formul\ae{}  (\ref{eqEqInt:trace}), the integral equation (\ref{eqEqInt:EqInt}) becomes
%Imposing to $\utm$ a homogeneous Neumann boundary condition of $\Gamma$ gives rise to the following Fredholm second kind integral equation
\begin{equation}\label{eq:MFIE}
\MFIED \rho = -  \duincg.
\end{equation}
%Note that the operator $A$ is here the interior normal trace $\gammaum$ on $\Gamma$.
This Fredholm second kind integral equation, named \emph{Magnetic Field Integral Equation} (MFIE), is well-posed and equivalent to the scattering problem (\ref{eqEqInt:ProblemeUt}) as far as $k$ is not an irregular Neumann frequency.
\begin{prop}\label{prop:MFIE}
If $k\not\in F_N(\Omegam)$, then the quantity $\Lop\rho+\uinc$ is the solution of the scattering problem (\ref{eqEqInt:ProblemeUt}) \ssi $\rho$ is the solution of the MFIE (\ref{eq:MFIE}). %In that case, we have $\rho = -\dn\ut|_{\Gamma}$.
\end{prop}

\begin{remark}
For every irregular frequency $k$ of $F_{N}(\Omegam)$, the operator $\MFIED$ is no more one-to-one. In that case and contrary to the EFIE, the single-layer potential $\Lop\rhot + \uinc$ based on a solution $\rhot$ of the MFIE is not guaranteed to be the solution of the scattering problem (\ref{eqEqInt:ProblemeUt}).
\end{remark}

\subsubsection{CFIE (\textit{Combined Field Integral Equation})}

To avoid the irregular frequencies problem, Burton and Miller \cite{BurMil70} considered a linear combination of the EFIE and the MFIE by imposing a Fourier-Robin boundary condition to $\utm$ on $\Gamma$:
$$
%(1-\alpha) \dn\utm + \alpha \eta \utm = 0,
A = (1-\alpha) \gammaum + \alpha \eta \gammazm,
$$
with
\begin{equation}\label{eq:condAlphaEta}
0 <\alpha <1 \qquad\text{ and } \qquad \Im(\eta) \neq 0,
\end{equation}
where $\Im(\eta)$ is the imaginary part of the complex number $\eta$. 
%Note that the operator $A$ is then a linear combination of the interior trace and normal trace :
%$$
%A = (1-\alpha) \gammaum + \alpha \eta \gammazm.
%$$
Hence, the boundary integral equation (\ref{eqEqInt:EqInt}) reads as
\begin{equation}\label{eq:CFIE}
\CFIED \rho = - \left[ (1-\alpha) \dn\uinc|_\Gamma + \alpha \eta \uinc|_\Gamma \right].
\end{equation}
This \emph{Combined Field Integral Equation} (CFIE, denomination of Harrington and Mautz \cite{HarMau78} in electromagnetism) or Burton-Miller integral equation \cite{BurMil70} is well-posed for every frequency $k$.
\begin{prop}\label{prop:CFIE}
For any $k>0$ and for any couple $\alpha$ and $\eta$ satisfying condition (\ref{eq:condAlphaEta}), the single-layer potential $\Lop\rho+\uinc$ is the solution of the scattering problem \ssi $\rho$ is the solution of the CFIE (\ref{eqEqInt:ProblemeUt}). %If so, the density $\rho$ satisfies $\rho = -\dn\ut|_{\Gamma}$.
\end{prop}

%\begin{remark}
%A study \cite{zdzd} on the choice of the parameters $\alpha$ and $\eta$ concludes that the numerical optimal choices are
%$$
%???
%$$
%\end{remark}

%----------------------------------------
\section{Single scattering preconditioned direct integral equations}\label{sec:SingleScat}
%----------------------------------------

%----------------------------------------
\subsection{Integral equation associated to $A$}
%----------------------------------------

Previous sections show that depending on the choice of the operator $A$, the boundary integral equation (\ref{eqEqInt:EqInt}) will be the EFIE (trace), the MFIE (normal trace) or the CFIE (linear combination). Therefore, in order to merge the notations, only the boundary integral equation (\ref{eqEqInt:EqInt}) will be considered, which can be rewritten as
\begin{equation}\label{eq:A}
\LA \rho = -A\uinc, \qquad \text{ with } \LA \rho = A\Lop\rho. 
\end{equation}
In this paper, this integral equation is called \emph{integral equation associated to $A$} or, in short, \emph{integral equation $A$}. It represents one of the three direct integral equation EFIE, MFIE or CFIE, depending on the choice of $A$.
As an example, for the EFIE, the operator $A$ is the interior trace $\gammazm$ and $\LA = \gammazm\Lop = L$ is the boundary single-layer integral operator. 
According to Properties \ref{prop:EFIE}, \ref{prop:MFIE} and \ref{prop:CFIE}, the integral equation associated to $A$ is uniquely solvable if %\ssi 
$k\not\in \FDN$.
\begin{prop}\label{prop:eqA}
Let us assume that $k\not\in F_{DN}(\Omegam)$ and that $A$ is whether the interior trace $\gammazm$, the interior normal trace $\gammaum$ or a linear combination  $(1-\alpha)\gammaum +\alpha\eta\gammazm$ such that $\alpha$ and $\eta$ satisfy relation (\ref{eq:condAlphaEta}). The quantity $\Lop\rho+\uinc$ is then the unique solution of the scattering problem (\ref{eqEqInt:ProblemeUt}) \ssi
 $\rho$ is the unique solution of the integral equation $A$ (\ref{eq:A}). 
%
%Moreover, in that case, the solution $\rho$ is independent of $A$ and is given by
%$$
%\rho = -\dn\utg.
%$$
\end{prop}

In what follows, irregular frequencies will be avoided by always assuming that $k\not\in\FDN$.

%----------------------------------------
%	\section{Multiple scattering}
%----------------------------------------
%%%%%%%%%
\subsection{Multiple scattering case}%Matrix form of the integral equation $A$ in the multiple scattering case}
%%%%%%%%%

The domain $\Omegam$ is now supposed to be a collection of $M$ disjoint bounded open sets $\Omegam_{p}$ of $\Rb^{d}$, $p=1,\ldots, M$, such that every domain $(\Rb^{d}\setminus\overline{\Omegam_{p}})$ is connected, as this is the case for the propagation domain $\Omegaps = \Rb^{d}\setminus\overline{\Omegam}$. 
In this paper, single scattering designates scattering in a medium containing only one scatterer whereas multiple scattering is used for a medium containing more than one obstacle. This article being focused on the multiple scattering case, $M$ will be assumed to satisfy $M\geq 2$.

As $\Omegam$ is composed of $M$ disjoint obstacles $\Omegamp$, $p=1,\ldots,M$, the single-layer volume integral operator $\Lop$ can be written as the sum of $M$ operators $\Lop_{q}$, $q=1,\ldots,M$, defined by
\begin{equation}\label{eqEqInt:Lopq}
\begin{array}{r c c  l}
\Lop_{q} : &   H^{-1/2}(\Gammaq) & \longrightarrow& H^{1}_{loc}(\Rb^{d}) \\
 & \rhoq &\longmapsto& \dsp{\Lopq\rhoq, \qquad\forall\xx\in\Rb^{d}, \quad \Lopq\rhoq(\xx) = \int_{\Gammaq} G(\xx,\yy)\rhoq(\yy)\;\dd\yy}.
 \end{array}
\end{equation}
Therefore the single-layer potential can be decomposed as follows
%if $\rhoq =\rho|_{\Gammaq}$ for any density $\rho \in H^{-1/2}(\Gamma)$ and index $q=1,\ldots,M$, we have
\begin{equation}\label{eq:decomposeLop}
\forall\rho\in H^{-1/2}(\Gamma),\qquad \Lop\rho = \sum_{q=1}^{M}\Lopq\rhoq, \qquad \text{ with } \rhoq = \rho|_{\Gammaq}.
\end{equation}
Now, for every $p=1,\ldots,M$, let $\Ap$ be the restriction of the operator $A$ to $\Gammap$:
\begin{equation}\label{eq:Ap}
\forall g \in H^{1}(\Omegam), \qquad 
\Ap g = (Ag)|_{\Gammap},
\end{equation}
%for every sufficiently smooth function $g$.
With these notations, the integral equation $A$ (\ref{eq:A}) satisfied by $\rho$ can be written equivalently as a system of $M$ integral equations
$$
\forall p=1,\ldots,M, \qquad \Ap\Lop\rho = -\Ap\uinc,
$$
or, using the decomposition \eqref{eq:decomposeLop}, as
\begin{equation}\label{eqEqInt:EqIntA1}
\forall p=1,\ldots,M, \qquad \sum_{q=1}^{M}\Ap\Lopq\rhoq = -\Ap\uinc.
\end{equation}
Finally, by introducing the operators $\LApq$, for $p,q=1,\ldots,M$, defined by %on $H^{-1/2}(\Gammaq)$ by
\begin{equation}\label{eqEqInt:OpLApq}
\forall \rhoq\in H^{-1/2}(\Gammaq),\qquad \LApq\rhoq = \Ap(\Lopq \rhoq),
\end{equation}
then the $M$ integral equations (\ref{eqEqInt:EqIntA1}) can be written in the following matrix form
$$
\left[\begin{array}{c c c c}
\LAuu & \LAud & \ldots & \LAuM \\
\LAdu & \LAdd & \ldots & \LAdM \\
\vdots & \vdots & \ddots & \vdots \\
\LAMu & \LAMd & \ldots & \LAMM \\
\end{array}\right]
\left[\begin{array}{c}
\rho_1 \\
\rho_{2} \\
\vdots \\
\rho_{M} \\
\end{array}\right]
= 
- \left[\begin{array}{c}
A_{1}\uinc \\
A_{2}\uinc \\
\vdots \\
A_{M}\uinc \\
\end{array}\right].
$$
Recall that this equation is just a matrix form of the integral equation associated to $A$ (\ref{eq:A}).
For now on, the operator $\LA$ will be identified to its associated matrix $(\LApq)_{1\leq p,q\leq M}$.

%%%%%%%%%
\subsection{Single scattering operator and preconditioned integral equation $A$}
%%%%%%%%%

%Let us now introduce the \emph{single scattering operator} $\LAsgl$, corresponding to the diagonal part of the operator $\LA$:
Let the \emph{single scattering operator} $\LAsgl$, corresponding to the diagonal part of the operator $\LA$, be defined by:
\begin{equation}\label{eq:LAsgl}
\LAsgl = \left[\begin{array}{c c c c}
\LAuu & 0 & \ldots & 0 \\
0 & \LAdd & \ldots & 0 \\
\vdots & \vdots & \ddots & \vdots \\
0 & 0 & \ldots & \LAMM \\
\end{array}\right].
\end{equation}
Indeed, each component $\LApp$ of $\LAsgl$ represents the self-interaction of the scatterer $\Omegamp$. More precisely, if the medium contains only one obstacle $\Omegamp$, with $p\in\{1,\ldots, M\}$, then this equality would hold true $\LA = \LApp$.

As the wavenumber $k$ is assumed not to be an irregular frequency of $\FDN$, then $k$ does also not belong to $F_{DN}(\Omegamp)$, for all $p=1,\ldots,M$, thanks to relations (\ref{eq:FDOmegamp}) and (\ref{eq:FDN}). 
Consequently and due to proposition \ref{prop:eqA}, for $p=1,\ldots, M$, the operator $\LApp$ associated to the single scattering is invertible. Thus, the single scattering operator $\LAsgl$ is also invertible with inverse operator
$$
\LAsgl^{-1} = \left[\begin{array}{c c c c}
(\LAuu)^{-1} & 0 & \ldots & 0 \\
0 & (\LAdd)^{-1} & \ldots & 0 \\
\vdots & \vdots & \ddots & \vdots \\
0 & 0 & \ldots & (\LAMM)^{-1} \\
\end{array}\right].
$$
The integral equation associated to $A$ is now preconditioned 
%We now precondition the integral equation associated to $A$ 
by $\LAsgl$ 
%to obtain 
which gives rise to
\emph{the preconditioned integral equation associated to $A$} (or in short \emph{the preconditioned integral equation $A$}):
\begin{equation}\label{eqEqINt:eqAprecond}
\LAsgl^{-1}\LA\rho = -\LAsgl^{-1}A\uinc,
\end{equation}
where the operator $\LAsgl^{-1}\LA$ has the following matrix form
\begin{equation}\label{eqEqINt:LAsglLA}
\LAsgl^{-1} \LA = \left[\begin{array}{c c c c}
I & (\LAuu)^{-1}\LAud & \ldots & (\LAuu)^{-1}\LAuM \\
(\LAdd)^{-1}\LAdu & I & \ldots & (\LAdd)^{-1}\LAdM \\
\vdots & \vdots & \ddots & \vdots \\
(\LAMM)^{-1}\LAMu & (\LAMM)^{-1}\LAMd & \ldots & I \\
\end{array}\right].
\end{equation}
Note that this preconditioning accelerates the convergence rate of an iterative solver, like the GMRES, as illustrated by the numerical example given in Section \ref{sec:num}.

%----------------------------------------
\subsection{Equality of the preconditioned direct integral equations}
%----------------------------------------

This section contains the main result of the paper, that is the equality between the three direct preconditioned integral equations. In other words, 
here is shown 
%we will show 
that the preconditioned integral equations associated to $A$ is independent of the choice of the operator $A$. As far as the author knows, this surprising result is new.

To prove this, a second ``general'' integral equation, called \emph{integral equation associated to $B$}, with $B\neq A$, is introduced and preconditioned in the same way as the integral equation associated to $A$:
\begin{equation}\label{eqEqINt:eqBprecond}
\LBsgl^{-1}\LB\rho = -\LBsgl^{-1}B\uinc,
\end{equation}
where $\LB = B\Lop$ and the operator $B$ is whether the interior trace $\gammazm$, the interior normal trace $\gammaum$ or a linear combination $(1-\alpha)\gammaum + \eta\alpha\gammazm$ with $\alpha$ and $\eta$ satisfying (\ref{eq:condAlphaEta}).
The operators $\LBsgl$ is defined in the same way as the operator $\LAsgl$ (see relation (\ref{eq:LAsgl})), with a formal change of $A$ by $B$. Moreover and similarly to relations (\ref{eq:Ap}) and (\ref{eqEqInt:OpLApq}), the operators $\Bp$ and $\LBpq$, $p,q=1,\ldots,M$, are introduced and defined by 
$$
\forall g\in H^{1}(\Omegam), \quad \Bp g =(Bg)|_{\Gammap}\qquad\text{ and }\qquad\forall\rhoq\in H^{-1/2}(\Gammaq),\quad \LBpq\rhoq = \Bp\Lopq\rhoq.
$$

Remark that, due to proposition \ref{prop:eqA}, for $k\not\in F_{DN}(\Omegam)$, the quantity $\Lop\rho + \uinc$ is the solution of the scattering problem (\ref{eqEqInt:ProblemeUt}) \ssi $\rho$ is the solution of the integral equation associated to $A$ or to $B$. In other words, provided $k\not\in \FDN$, the two integral equations $A$ and $B$ are equivalent and have the same solution.

The aim of this section is to prove that the preconditioned integral equation associated to $A$ (\ref{eqEqINt:eqAprecond}) is exactly the same as the preconditioned integral equation associated to $B$ (\ref{eqEqINt:eqBprecond}). Since the unknown $\rho$ is the same in both equations, it is sufficient to prove that the operators $\LAsgl^{-1}\LA$ and $\LBsgl^{-1}\LB$ are identical. 
More precisely and thanks to their matrix form (\ref{eqEqINt:LAsglLA}), it is sufficient to show that the following equality holds true
\begin{equation}\label{eqEqInt:LApqLBpq}
\forall p,q=1,\ldots,M,\qquad (\LApp)^{-1}\LApq = (\LBpp)^{-1}\LBpq.
\end{equation}
Furthermore, for $p=q$, the above equality is obvious since
$$
(\LApp)^{-1}\LApp = (\LBpp)^{-1}\LBpp = \Ip,
$$
where $\Ip$ is the identity operator on $H^{-1/2}(\Gammap)$.
Therefore, equality (\ref{eqEqInt:LApqLBpq}) needs to be shown only for $p\neq q$. 
Let the following result be first established.
%Before doing it, we first establish the following result.
\begin{lemma}\label{lemme:ABpp}
Assuming that $k\not\in F_{DN}(\Omegam)$ then the following equality holds true
$$
\forall p=1,\ldots,M, \qquad (\LApp)^{-1}\Ap = (\LBpp)^{-1}\Bp.
$$
\end{lemma}
\begin{proof}
Let $p=1,\ldots, M$ and $g$ be an element of $H^{1}_{loc}(\Rb^{d})$. The single scattering problem associated to the obstacle $\Omegamp$ is to find the scattered field $v$ solution of
\begin{equation}\label{eqPrecond:diffsimple}
\begin{cases}
(\Delta + k^{2}) v = 0 & \text{in } \Rb^{d} \setminus\overline{\Omegamp},\\
v = - g & \text{on } \Gammap,\\
v \text{ outgoing.}
\end{cases}
\end{equation}
The unique solution of this problem is the single-layer potential $\Lopp\rhop$, where $\rhop\in H^{-1/2}(\Gamma_p)$ is indifferently obtained by solving the integral equation associated to $A$
$$
\LApp\rhop = -\Ap g,
$$
or the one associated to $B$
$$
\LBpp\rhop = -\Bp g.
$$
These two equations being well-posed ($k\notin F_{DN}(\Omegam)$), we have
$$
(\LApp)^{-1}\Ap g = (\LBpp)^{-1}\Bp g.
$$
The proof is ended by virtue of the arbitrariness of $g$.
\end{proof}

The main result of this section can now be established.
\begin{theorem}\label{theo:equalityBIE}
If $k\not\in F_{DN}(\Omegam)$, then the operators $\LAsgl^{-1}\LA$ and $\LBsgl^{-1}\LB$ are equal. In other words, the operator $\LAsgl^{-1}\LA$ does not depend on the choice of the operator $A$.
\end{theorem}
\begin{proof}
It is sufficient to show that, for $p,q=1,\ldots,M$, with $p\neq q$,
$$
(\LApp)^{-1}\LApq = (\LBpp)^{-1}\LBpq.
$$
Recalling that the operator $\LApq$ is given by $\LApq = \Ap\Lopq$ and using lemma \ref{lemme:ABpp}, it appears that
$$
(\LApp)^{-1}\LApq = (\LApp)^{-1}\Ap\Lopq = (\LBpp)^{-1}\Bp\Lopq.
$$
Therefore, applying the definition $\Bp\Lopq = \LBpq$ gives rise to the sought equality
$$
(\LApp)^{-1}\LApq =(\LBpp)^{-1}\LBpq, \qquad \qquad\forall p,q=1,\ldots,M,
$$
and the proof is ended.
\end{proof}

As a conclusion, let us point out that the preconditioned integral equations $A$ and $B$ have the same operator ($\LAsgl^{-1}\LA = \LBsgl^{-1}\LB$) and the same solution $\rho$. Consequently, their right hand sides are also equal:
$$
- \LAsgl^{-1}\uinc = - \LBsgl^{-1}\uinc.
$$
Thus, the preconditioned integral equation $A$
$$
\LAsgl^{-1}\LA \rho = - \LAsgl^{-1}\uinc,
$$
and the preconditioned integral equation $B$
$$
\LBsgl^{-1}\LB \rho = - \LBsgl^{-1}\uinc,
$$
are exactly the same. As a consequence, preconditioning any direct integral equation EFIE, MFIE or CFIE with its single scattering operator will lead to exactly the same equation. Obviously, the frequency $k$ must not be an irregular one, because in that case the single scattering operator $\LAsgl$ could be no more invertible (especially for the EFIE and the MFIE).
Finally, it should be pointed out that the CFIE is well-posed for every wavenumber $k>0$ and thus, the preconditioned CFIE is also well-posed for all $k>0$.

\begin{remark}
Theorem \ref{theo:equalityBIE} is written for the three direct integral equations EFIE, MFIE and CFIE, but can be easily extended to any other direct boundary integral equation, for which the scattered field is also given by $u = \Lop\rho$, with $\rho = -\dn\ut|_{\Gamma}$. In fact, theorem \ref{theo:equalityBIE} can be proved for any couple of boundary integral equations, provided that they share the same expression of the scattered field and that they are based on the exact same unknown density (\textit{e.g.} single-layer potential of density $\rho= -\dn\ut|_{\Gamma}$ for the direct integral equations). In other words, with these assumptions and after being preconditioned by their single scattering operator, these set of boundary integral equations will become identical to each other.
\end{remark}
%$$
%\ut = \Top \varphi + \uinc.
%$$
%Operator $\Top$ was the single-layer potential operator $\Lop$ for the direct integral equation. Let now $A$ and $B$ be two trace-like operators as previously defined, and let two integral equation $A$ and $B$ be defined as
%\begin{equation}\label{eq:DBIEGen}
%\TA\psi = -A\uinc \quad\text{ and } \quad\TB\psi = -B\uinc,
%\end{equation}
%with
%$$
%\TA = A\Top \quad\text{ and } \quad\TB = B\Top.
%$$
%As previously done, let $\TAsgl$ and $\TBsgl$ be the single scattering operator of $\TA$ and $\TB$. The general result can now be established
%\begin{prop}
%Let $\Top$ be a linear combination of the single- and double-layer operator $\Lop$ and $\Mop$, and $\varphi$ be a density in $H^{1/2}(\Gamma)$. Let $A$ and $B$ be two trace-like operator so that the two boundary integral equations (\ref{eq:DBIEGen}) satisfy:
%\begin{enumerate}
%\item Equations (\ref{eq:DBIEGen}) is uniquely solvable for a set of frequency $I$.
%\item For $k\in I$, the quantity $\Top\psi + \uinc$ solves the scattering problem (\ref{eqEqInt:ProblemeUt}) if and only if $\psi$ is the unique solution of one of the two equation of (\ref{eq:DBIEGen}).
%\end{enumerate}
%Then, for any $k\in I$, the following equality holds
%$$
%\TAsgl^{-1} \TA = \TBsgl^{-1}\TB.
%$$
%\end{prop}
%\begin{proof}
%This Proposition is a generalization of Proposition \ref{theo:equalityBIE}, which is proven by Lemma \ref{lemme:ABpp}. 
%\end{proof}
%}
%----------------------------------------
\section{Brakhage-Werner indirect integral equation}\label{sec:BW}
%----------------------------------------

The above results are now extended to the Brakhage-Werner indirect integral equation \cite{BraWer65}. In fact, after being preconditioned by its single scattering operator, the Brakhage-Werner integral equation does not lead to the same equation as the one obtained with the direct integral equations. %However they are equal up to an invertible operator. 
However they are similar. 
%Numerically, this will make the iterative solver have the same behavior for any preconditioned integral equation.
This section begins by recalling how the Brakhage-Werner integral equation can be obtained and after that the result is established by comparison with the EFIE.

\subsection{Brakhage-Werner indirect integral equation}

This paragraph begins with some notations. The total field $\ut$ is here sought as a linear combination of a single- and a double-layer of density $\psi\in H^{1/2}(\Gamma)$:
$$
\ut = \uinc + \Lop_{BW}\psi,
$$
where the operator $\Lop_{BW}$ of parameter $\eta_{BW}$ is given by
%a combination of the single- and the double-layer operators:
\begin{equation}\label{eq:LopBW}
\forall\xx\in\Rb^{d}\setminus\Gamma, \qquad \Lop_{BW}\psi(\xx) = (-\eta_{BW}\Lop - \Mop)\psi(\xx) = 
\int_{\Gamma}\Big(\dny G(\xx,\yy) - \eta_{BW} G(\xx,\yy)\Big)\psi(\yy)\;\dd\Gamma(\yy),
\end{equation}
with $\Im(\eta) \neq 0$. The integral equation is obtained by applying the exterior trace $\gammazps$ on $\Gamma$ to $\ut$. Indeed, the Dirichlet boundary condition $\gammazps \ut = 0$ and the traces relation (\ref{eqEqInt:trace}) directly give the Brakhage-Werner integral equation solved by $\psi$ 
\begin{equation}\label{eqEqInt:BW}
\LBW \psi = -\uincg,
\end{equation}
with
$$
\LBW = \left( -\eta L -M + \frac{1}{2}I \right).
$$
This second kind integral equation does not suffer from irregular frequency \cite{BraWer65}.
\begin{prop}
For all $k>0$, the quantity $\Lop_{BW}\psi+\uinc$ is the solution of the scattered field (\ref{eqEqInt:ProblemeUt}) \ssi $\psi$ is the solution of the Brakhage-Werner integral equation (\ref{eqEqInt:BW}).
\end{prop}
\begin{remark}
Other generalizations of these equations, when $\eta_{BW}$ is an operator, are available for example in \cite{AloBorLev07, AntoineDarbasQJMAM, AntoineDarbasM2AN}
\end{remark}
\begin{remark}\label{rem:BW}
A numerical study concerning the optimal choice of parameter $\eta_{BW}$, appearing in relation (\ref{eq:LopBW}), is proposed in \cite{KreSpa83} in the case of a single spherical or circular obstacle of radius $R$. For a Dirichlet boundary condition, the choice $\eta_{BW} = i/2\max(1/R,k)$ leads to a reasonable condition number of the matrix of the linear system associated to the Brakhage-Werner integral equation, for sufficiently high frequency. 
Recent works have been done on how to choose this parameter for much more general domains, see for example \cite[\S6]{ChaGraLan09} and \cite[\S5.1]{ChaGraLan12} for the case of large $k$ and \cite[\S2.6 and \S2.7]{BetChaGra11} for the case of small frequency $k$. Note also that, according to \cite[Remark 2.24]{ChaGraLan12}, these results apply to both $L_{BW}$ and the CFIE operator, since when $\alpha = 1/2$, these operators are adjoints (up to a factor of $1/2$) in the real $L^2$ inner product.
\end{remark}
%\begin{remark}\label{rem:BW}
%A numerical study concerning the optimal choice of parameter $\eta_{BW}$, appearing in relation (\ref{eq:LopBW}),  is proposed in \cite{KreSpa83} \textcolor{red}{in the case of a single spherical or circular obstacle}. In the single scattering case with a Dirichlet boundary condition, the choice $\eta_{BW} = i/2\max(1,k)$ leads to a reasonable condition number of the matrix of the linear system associated to the Brakhage-Werner integral equation. Other generalizations of these equations, when $\eta_{BW}$ is an operator, are available for example in \cite{AloBorLev07, AntoineDarbasQJMAM, AntoineDarbasM2AN}.
%\textcolor{red}{Recent works have been done on how to choose this parameter for much more general domains, see for example \cite[\S6]{ChaGraLan09} and \cite[\S5.1]{ChaGraLan12} for the case of large $k$ and \cite[\S2.6 and \S2.7]{BetChaGra11} for the case of small frequency $k$. Note also that these results apply to both $L_{BW}$ and the CFIE operator, since when $\alpha = 1/2$ these operators are adjoints (up to a factor of $1/2$) in the real-$L^2$ inner product (see \cite[Remark 2.24]{ChaGraLan12}).}
%\end{remark}
As in the previous section, $M$ volume integral operators $\Lop_{BW}^{q}$, for $q=1,\ldots,M$, are introduced and defined for every density $\psiq \in H^{1/2}(\Gammaq)$ by
$$
\forall\xx\in\Rb^{d}\setminus\Gammaq,\qquad\Lop_{BW}^{q} \psiq(\xx) = \int_{\Gammaq}\Big(\dny G(\xx,\yy) - \eta G(\xx,\yy)\Big)\psiq(\yy)\;\dd\Gamma(\yy).
$$
Thus, the potential $\dsp{\Lop_{BW}\psi}$ reads as
$$
\forall \psi\in H^{1/2}(\Gamma), \qquad \Lop_{BW}\psi = \sum_{q=1}^{M} \Lop_{BW}^{q}\psiq, \qquad \text{ with } \psiq = \psi|_{\Gammaq}.
$$
Finally, for $p,q=1,\ldots,M$, the operator $\LBWpq$ is defined by
\begin{equation}\label{eqEqInt:gammapLopBWp}
\forall\psiq\in H^{1/2}(\Gammaq), \qquad \LBWpq \psiq = \left.\left(\Lop_{BW}^{q} \psiq \right)\right|_{\Gammap}.
\end{equation}

The Brakhage-Werner integral equation (\ref{eqEqInt:BW}) can now be written in the following matrix form
\begin{equation}\label{eqEqInt:LBW}
\left[\begin{array}{c c c c}
\LBW^{1,1} & \LBW^{1,2} & \ldots & \LBW^{1,M} \\[0.2cm]
\LBW^{2,1} & \LBW^{2,2} & \ldots & \LBW^{2,M}\\[0.2cm]
\vdots & \vdots & \ddots & \vdots \\[0.2cm]
\LBW^{M,1} & \LBW^{M,2} & \ldots & \LBW^{M,M} \\
\end{array}\right]
\left[\begin{array}{c}
\psi_1 \\[0.2cm]
\psi_{2} \\[0.2cm]
\vdots \\[0.2cm]
\psi_{M} \\
\end{array}\right]
= 
- \left[\begin{array}{c}
\uinc|_{\Gamma_{1}} \\[0.2cm]
\uinc|_{\Gamma_{2}} \\[0.2cm]
\vdots \\[0.2cm]
\uinc|_{\Gamma_{M}} \\
\end{array}\right],
\end{equation}
and the single scattering operator $\LBWsgl$ of the Brakhage-Werner integral equation reads as
$$
\LBWsgl = 
\left[\begin{array}{c c c c}
\LBW^{1,1} & 0 & \ldots & 0 \\[0.2cm]
0 & \LBW^{2,2} & \ldots & 0\\[0.2cm]
\vdots & \vdots & \ddots & \vdots \\[0.2cm]
0 & 0 & \ldots & \LBW^{M,M} \\
\end{array}\right].
$$
Because of the well-posedness of the Brakhage-Werner integral equation for every frequency $k$, each operator $\LBW^{p,p}$ is invertible.
%\begin{lemma}
%For $p=1,\ldots,M$, the operator $\LBWpp$ is invertible.
%\end{lemma}
%\begin{proof}
%Considérons le problème de diffraction simple ne comportant que l'obstacle $\Omegamp$ :
%$$
%\begin{cases}
%(\Delta + k^{2})\vt^{p} = 0 & (\Rb^{d}\setminus\overline{\Omegamp})\\
%\vt^{p} = 0 & (\Gammap) \\
%(\vt^{p} - \uinc) \text{ sortant}.
%\end{cases}
%$$
%La proposition \ref{propEqInt:BWD} implique que le potentiel $(\Lop_{BW}\xi_{p} + \uinc)$ est solution du problème de diffraction simple ci-dessus si et seulement si $\xi_{p}$ est solution de l'équation intégrale de Brakhage-Werner
%$$
%\LBWpp \xi_{p} = -\uinc|_{\Gammap}.
%$$
%Ainsi, l'opérateur $\LBWpp$ est inversible.
%\end{proof}
%Ce lemme implique alors que l'opérateur $\LBWsgl$ est inversible et son opérateur inverse $\LBWsgl^{-1}$ est donné par
Thus, the operator $\LBWsgl$ is invertible with inverse% $\LBWsgl^{-1}$ given by
$$
\LBWsgl^{-1} = 
\left[\begin{array}{c c c c}
(\LBW^{1,1})^{-1} & 0 & \ldots & 0 \\
0 & (\LBW^{2,2})^{-1} & \ldots & 0\\
\vdots & \vdots & \ddots & \vdots \\
0 & 0 & \ldots & (\LBW^{M,M})^{-1} \\
\end{array}\right].
$$
The Brakhage-Werner integral equation (\ref{eqEqInt:BW})  can now be composed on its left by the operator $\LBWsgl^{-1}$ to obtain the \emph{preconditioned Brakhage-Werner integral equation}
\begin{equation}\label{eqEqInt:BWDiag}
\LBWsgl^{-1}\LBW\psi = -\LBWsgl^{-1}(\uincg).
\end{equation}

\subsection{EFIE}

%Let us consider again the EFIE, given by equation (\ref{eq:EFIE}). 
Let the EFIE, given by equation (\ref{eq:EFIE}), be considered again.
For this integral equation, the total field $\ut$ is written as a single-layer potential with density $\rho\in H^{-1/2}(\Gamma)$:
$$
\ut  =\uinc + \Lop\rho.
$$
The EFIE can be obtained through (at least) two possibilities. The first one, considered in the previous section, consists in introducing a fictitious interior field $\utm = \Lop\rho + \uinc$ in $\Omegam$ and applying to it a homogeneous Dirichlet boundary condition on $\Gamma$. Another possibility is to apply directly the Dirichlet boundary condition $\ut|_{\Gamma}  = 0$ to the quantity $\uinc + \Lop\rho$. Thanks to the continuity of the single-layer potential through $\Gamma$ (Proposition \ref{propEqInt:trace}), this gives directly the electric field integral equation:
\begin{equation}\label{eqEqInt:EFIEbis}
L\rho = - \uincg.
\end{equation}
As a consequence, both the EFIE and the Brakhage-Werner integral equation are obtained through a direct application of the exterior trace $\gammazps$ to the total field $\ut$. This common point is the key to prove that, after being preconditioned by their single scattering operator, the EFIE and the Brakhage-Werner integral equation are 
%equal up to an invertible operator.
similar.
%In what follows, we assume that $k\not\in\FD$ and hence that the operator $L$ is invertible. 
In what follows, the wavenumber $k$ is assumed to satisfy $k\not\in\FD$, which implies that the operator $L$ is invertible. 
The volume single-layer integral operator $\Lop$ is now decomposed into $M$ operators $(\Lopq)_{1\leq q\leq M}$, as previously (see relation (\ref{eqEqInt:Lopq})),
%and we introduce the operators $(\Lpq)_{1\leq p,q\leq M}$ defined by
which leads to introduce the operators $(\Lpq)_{1\leq p,q\leq M}$ defined by
\begin{equation}\label{eqEqInt:gammapLop}
\forall \rhoq\in H^{-1/2}(\Gammaq),\qquad \Lpq\rhoq = \left.\left(\Lopq\rhoq\right)\right|_{\Gammap}.
\end{equation}
With these notations, the EFIE has the following matrix form
\begin{equation}\label{eqEqInt:L}
\left[\begin{array}{c c c c}
L^{1,1} & L^{1,2} & \ldots & L^{1,M} \\
L^{2,1} & L^{2,2} & \ldots & L^{2,M}\\
\vdots & \vdots & \ddots & \vdots \\
L^{M,1} & L^{M,2} & \ldots & L^{M,M} \\
\end{array}\right]
\left[\begin{array}{c}
\rho_{1} \\
\rho_{2} \\
\vdots \\
\rho_{M} \\
\end{array}\right]
= 
- \left[\begin{array}{c}
\uinc|_{\Gamma_{1}} \\
\uinc|_{\Gamma_{2}} \\
\vdots \\
\uinc|_{\Gamma_{M}} \\
\end{array}\right],
\end{equation}
and its single scattering operator $\Lsgl$ is given by
$$
\Lsgl = 
\left[\begin{array}{c c c c}
L^{1,1} & 0 & \ldots & 0 \\
0 & L^{2,2} & \ldots & 0\\
\vdots & \vdots & \ddots & \vdots \\
0 & 0 & \ldots & L^{M,M} \\
\end{array}\right].
$$
According to theorem \ref{theo:FreqIr}, the operator $\Lsgl$ is invertible. Therefore, the EFIE can be left composed by $\Lsgl^{-1}$ which gives rise to the preconditioned EFIE
\begin{equation}\label{eqEqInt:EFIEDiag}
\Lsgl^{-1}L\rho = -\Lsgl^{-1}\uincg,
\end{equation}
with
$$
\Lsgl^{-1} = 
\left[\begin{array}{c c c c}
(L^{1,1})^{-1} & 0 & \ldots & 0 \\
0 & (L^{2,2})^{-1} & \ldots & 0\\
\vdots & \vdots & \ddots & \vdots \\
0 & 0 & \ldots & (L^{M,M})^{-1} \\
\end{array}\right].
$$

%-----------------------------------------------------------------------------------------
\subsection{Relation between preconditioned EFIE and preconditioned Brakhage-Werner integral equation}
%-----------------------------------------------------------------------------------------

In this section is proved that the operators $\Lsgl^{-1}L$ and $\LBWsgl^{-1}\LBW$ are 
%equal, up to an invertible operator. 
similar. 
To achieve this, the following lemma must first be established.
\begin{lemma}\label{lemmeEqInt:LoppLpp}
When $k\not\in F_{D}(\Omegam)$, then for $p=1,\ldots, M$, the below equality holds true
$$
\Lopp(\Lpp)^{-1} = \Lop_{BW}^{p}(\LBWpp)^{-1}.
$$
\end{lemma}
\begin{proof}
For $p=1,\ldots,M$, consider the scatterer $\Omegamp$, an element $f_{p}$ of $H^{1/2}(\Gammap)$ and the following single scattering problem of unknown field $v$
\begin{equation}\label{eqEqIntdem:systemp}
\begin{cases}
(\Delta + k^{2})v = 0 &\text{in } \Rb^{d}\setminus\overline{\Omegamp}, \\
v = -f_{p}& \text{on }\Gammap, \\
v \text{ outgoing.}
\end{cases}
\end{equation}
Let $\omega$ and $\omega_{BW}$ be two potentials defined by
$$
\left\{\begin{array}{l}
\omega = \Lopp\left[(\Lpp)^{-1}f_{p}\right],\\[0.2cm]
\omega_{BW} = \Lop_{BW}^{p}\left[(\LBWpp)^{-1}f_{p}\right].
\end{array}\right.
$$
As a linear combination of single- and double-layer potential, these two functions are radiating solution of the Helmholtz equation in $(\Rb^{d}\setminus\overline{\Omegamp})$ (see proposition \ref{propEqInt:potentiel}). Moreover, applying the exterior trace on $\Gammap$ to $\omega$ and $\omega_{BW}$ leads to
$$
\left\{\begin{array}{l}
\omega|_{\Gammap} = -\left. (\Lopp(\Lpp)^{-1}f_{p})\right|_{\Gammap} = -\Lpp(\Lpp)^{-1}f_{p} = -f_{p},\\[0.3cm]
\omega_{BW}|_{\Gammap} = -\left.(\Lop_{BW}^{p}(\LBWpp)^{-1}f_{p})\right|_{\Gammap} = -\LBWpp(\LBWpp)^{-1}f_{p} = -f_{p}.
\end{array}\right.
$$
Thus, the waves $\omega$ and $\omega_{BW}$ are both solutions of the single scattering problem (\ref{eqEqIntdem:systemp}). By unicity (see theorem \ref{theo:UniqueSolution}), it appears that
$$
\forall f_{p} \in H^{1/2}(\Gammap),\qquad \Lopp(\Lpp)^{-1}f_{p} = \Lop_{BW}^{p}(\LBWpp)^{-1}f_{p}.
$$
\end{proof}

A second lemma can now be obtained.
%, which can be seen a generalization the previous one to the multiple scattering case, can now be expressed
\begin{lemma}\label{lemmeEqInt:LLsgLBWLBWsgl}
For $k\not\in F_{D}(\Omegam)$, the following equality holds true
$$
L\Lsgl^{-1} = \LBW \LBWsgl^{-1}.
$$
\end{lemma}
\begin{proof}
Thanks to the matrix form of the operators (see relations (\ref{eqEqInt:L}) and (\ref{eqEqInt:LBW})), it is more convenient to prove that
\begin{equation}\label{eqEqIntdem:propBW2}
\forall p,q=1,\ldots,M,\qquad L^{q,p}(L^{p,p})^{-1} = \LBWqp(\LBWpp)^{-1}.
\end{equation}
%Let $p$ and $q$ be two indices such that $1\leq p,q\leq M$. 
%For $q=p$, we have
%$$
%(L \Lsgl^{-1})^{p,p} = (\LBW \LBWsgl^{-1})^{p,p} (=I). 
%$$
%As a consequence, it is sufficient to prove the assertion (\ref{eqEqIntdem:propBW2}) only for $p\neq q$. 
For two indices $1\leq p,q\leq M$, lemma \ref{lemmeEqInt:LoppLpp} implies that
$$
\Lopp(\Lpp)^{-1} = \Lop_{BW}^{p}(\LBWpp)^{-1}.
$$
Applying the exterior trace on $\Gammaq$ to this equality and using relations (\ref{eqEqInt:gammapLop}) and (\ref{eqEqInt:gammapLopBWp}) lead to
$$
L^{q,p}(L^{p,p})^{-1} = \LBWqp(\LBWpp)^{-1},
$$
which ends the proof.
%The sough relation (\ref{eqEqIntdem:propBW2}) is then proved for $p\neq q$.
\end{proof}

For now on, $\rho$ denotes the (unique) solution of the EFIE (\ref{eqEqInt:EFIEbis}) and $\psi$ the one of the Brakhage-Werner integral equation (\ref{eqEqInt:BW}). In other words, this means that%we have
$$
L\rho = -\uincg \qquad \text{ and }\qquad\LBW\psi = -\uincg.
$$
Consequently, the densities $\rho$ and $\psi$ are linked through the following relation
$$
\rho = L^{-1}\LBW\psi. 
$$
The next proposition shows that the invertible operator $L^{-1}\LBW$ is the transformation operator between the preconditioned EFIE and the preconditioned Brakhage-Werner integral equation.
\begin{theorem}\label{theo:equalityBW}
For every $k\not\in F_D(\Omegam)$, the operators $\Lsgl^{-1}L$ of the preconditioned EFIE and the operator $\LBWsgl^{-1}\LBW$ of the preconditioned Brakhage-Werner integral equation are similar in the sense that
$$
\left(L^{-1}\LBW\right)^{-1} \Lsgl^{-1}L \left(L^{-1}\LBW\right) = \LBWsgl^{-1} \LBW.
$$
\end{theorem}
\begin{proof}
The proof begins 
%We start 
with the equality given by lemma \ref{lemmeEqInt:LLsgLBWLBWsgl}:
$$
L\Lsgl^{-1} = \LBW\LBWsgl^{-1}.
$$
Composing the above equation on the left by $\LBW^{-1}$ implies that
$$
\left(L^{-1}\LBW\right)^{-1}\Lsgl^{-1} = \LBWsgl^{-1}.
$$
Now, the identity operator $I = LL^{-1}$ is introduced in the left hand side to obtain
$$
\left(L^{-1}\LBW\right)^{-1}\Lsgl^{-1}LL^{-1} = \LBWsgl^{-1}.
$$
Finally, it suffices to compose to the right by the invertible operator $\LBW$ to derive the sought relation
$$
\left(L^{-1}\LBW\right)^{-1}\Lsgl^{-1}L\left(L^{-1}\LBW\right) = \LBWsgl^{-1}\LBW.
$$
\end{proof}

Consequently, the operator of the preconditioned EFIE and the preconditioned Brakhage-Werner integral equation are 
similar.
%equal up to the invertible transmission operator $L^{-1}\LBW$. 
Due to theorem \ref{theo:equalityBIE}, this result can be extended to the preconditioned MFIE and preconditioned CFIE. 

\begin{remark}
So far, only the indirect integral equation of Brakhage-Werner has been studied. However, theorem \ref{theo:equalityBW} can be extended for any couple of boundary integral equations, provided that they are obtained by applying directly the Dirichlet boundary condition to the total field. This extended result can be proven in the same way as theorem \ref{theo:equalityBW} and summarized as follows for two general integral formulations. For $j=1,2$, let $\Top_{j}$ be a volume integral operator, such as the single-layer operator $\Lop$, the double-layer operator $\Mop$ or any linear combination of them, even involving an invertible operator such as in \cite{AntoineDarbasM2AN}. For $j=1,2$, the total field of the integral formulation numbered by $j$ is sough as $\ut = \Top_{j}\varphi_{j} + \uinc$, where $\varphi_{j}$ is the unknown density, solution of the boundary integral equation $j$:
\begin{equation}\label{eq:Tj}
T_{j}\varphi_{j} = -\gammazps\uinc, \qquad \text{with $T_{j} = \gammazps\Top_{j}$.}
\end{equation}
Hence, if the two integral equations (\ref{eq:Tj}) are uniquely solvable and equivalent to the scattering problem, then
the following equality holds true
$$
(T_{1}^{-1}T_{2})^{-1} \Tsgl_{1}^{-1}T_{1}  (T_{1}^{-1}T_{2}) = \Tsgl_{2}^{-1}T_{2},
$$
where,  for $j=1,2$,  $\Tsgl_{j}$ is the single scattering operator of $T_{j}$ (defined in the same way as above). In other words, the two preconditioned boundary integral equations are 
similar.
%equal, up to the invertible operator $(T_{1}^{-1}T_{2})$.
\end{remark}

\section{Numerical results}\label{sec:num}

Provided $k$ is not an irregular frequency and after being preconditioned, the EFIE, MFIE, CFIE and the Brakhage-Werner integral equation are equal or similar for the last case. As a consequence, their operator have the same spectrum. Numerically, this also implies that the behavior of an iterative solver will be the same for every integral equation. The numerical results presented in this section aim to illustrate these properties. They also aim to show that the single scattering preconditioner accelerate the convergence rate of the iterative solver. 
In our examples, the boundary integral equation are discretized thanks to the boundary element method, which is briefly described in the first paragraph only for the EFIE, the method being
similar for the other integral equations.
The second paragraph is devoted to the numerical results.%resolution of the derived linear system with the iterative solver GMRES and in particular in the behavior of the iterative solver for the preconditioned integral equations.

%%%%%%%%%%%
\subsection{Boundary element method: the example of the EFIE}
%%%%%%%%%
Recall that the EFIE for the Dirichlet boundary value problem reads as
$$
L\rho = - \uincg.
$$
Let $\Gamma$ be described by a polygonal approximation $\Gamma_{h}$ containing  $N_{h}$ segments.
%Let $\Gamma_{h}$ be a polygonal approximation of the boundary $\Gamma$ with $N_{h}$ segments.
%us introduce a polygonal approximation $\Gamma_{h}$ of the boundary $\Gamma$ containing $N_{h}$ segments. 
The largest length of the segments will be denoted by $h$. The finite element approximation space $V_h$ of $L^2(\Gamma_h)$ contains all the continuous  and piecewise linear functions on $\Gamma_h$ ($\mathbb{P}_1$ finite elements):
$$%\begin{equation}
V_h := \left\{ \rho_h \in \mathcal{C}^0(\Gamma_h) / \rho_h|_{K_j} \in \mathbb{P}_1, 1 \leq j \leq N_h \right\}.
$$%\end{equation}
The weak formulation of the EFIE is then given by
$$%\begin{equation}
\left\{
\begin{array}{l}
\textrm{Find $\rho_h  \in V_h$ such that} \\
\forall \Phi_h \in V_h, \hspace{0.2cm}\PSGammah{L_h\rho_h}{\Phi_h} = 
-\PSGammah{\uinc|_{\Gamma_h}}{\Phi_h}.
\end{array}
\right.
$$%\end{equation}
After approximating the double integrals using quadrature formul\ae{}, the weak formulation have then the following matrix form
$$%\begin{equation}\label{eqEqInt:SystFEM}
\left\{
\begin{array}{l}
\textrm{Find $\mathbf{\Rho}_h  \in \mathbb{C}^{N_h}$ such that} \\
\left[L_h \right] \Rho_h= -[M_h]\textbf{u}^{\textrm{inc}},
\end{array}
\right.
$$%\end{equation}
where $\left[L_h \right] \in \mathcal{M}_{N_h,N_h}(\mathbb{C})$ is the matrix of the single-layer potential, $\left[M_h \right]
\in \mathcal{M}_{N_h,N_h}(\mathbb{C})$ is the mass matrix for the linear finite elements,
$\Rho_h \in \mathbb{C}^{N_h}$ the nodal vector of the density and $\textbf{u}^{\textrm{inc}}\in \mathbb{C}^{N_h}$ the incident nodal vector. Then, the matrix of the single scattering $[\hat{L}_{h}]$ is computed by extracting and inverting the blocks located on the diagonal of the matrix $[L_{h}]$. Finally, the system approaching the preconditioned EFIE is given by
\begin{equation}\label{eqEqInt:SystFEM2}
\left\{
\begin{array}{l}
\textrm{Find $\mathbf{\Rho}_h  \in \mathbb{C}^{N_h}$ such that} \\
\big[\hat{L}_{h}\big]^{-1}\left[L_h \right] \Rho_h= -\big[\hat{L}_{h}\big]^{-1}[M_h]\textbf{u}^{\textrm{inc}}.
\end{array}
\right.
\end{equation}
The method is exactly the same for the other integral equations. 
In the example that follows, the single scattering preconditioner is obtained using a direct method, in the sense that the matrix of the linear system is fully computed and numerically inverted using a direct solver. However, in practice and as highlighted in the Introduction, the size of the linear system could be too large to compute and invert the matrix (\eg $\big[\hat{L}_{h}\big]$ in equation \eqref{eqEqInt:SystFEM2}), especially at high frequency and/or with a large number of obstacles. In that case, an iterative solver coupled with a fast and low cost matrix-vector product must be used, such as Fast Multipole Method. The single scattering preconditioning could then be applied considering another iterative procedure. Nevertheless and for the sake of clarity, the preconditioner is here still denoted is by $\big[\hat{L}_{h}\big]^{-1}$.

Lastly, computing the single scattering preconditioner involves the inverse of the single scattering matrix ($\big[\hat{L}_{h}\big]$ in the above example). Numerically, this operation is strongly dependent on the conditioning of the matrix and thus of the associated integral equation. Therefore, even if theorems \ref{theo:equalityBIE} and \ref{theo:equalityBW} show that the preconditioned integral equations are exactly the same or similar for the four integral equations, a well conditioned integral equation is still needed. %\end{remark}

%----------------------------------------
		\subsection{Numerical example}
%----------------------------------------

Three different kinds of scatterers are considered: ellipsoidal, rectangular and ``kite-shaped'' (see Figure \ref{figEqInt:FEMKite} for the last shape). 
A total of $30$ obstacles, $10$ from each different shape, are randomly distributed in the box $[0,60]\times[0,60]$ with a random characteristic size of order $1$. 
The distance $b_{pq}$ between the centers of two obstacles $\Omegamp$ and $\Omegamq$ is such that $b_{pq}\geq 3$.
The direction of the incident plane wave is set to $\Beta = (\cos(\pi/2),\sin(\pi/2))$ and the wavenumber $k$ to $20$. Finally, the mesh is generated using $15$ points per wavelength, implying that every obstacle is represented by around $300$ segments. An example of a geometry is shown on Figure \ref{figEqInt:FEM_obstacles}. The parameters $\alpha$ and $\eta$ of the CFIE are chosen with respect to what appears as a reasonable choice \cite[section 5.2.1]{Dar04} (other choices would be more optimal as highlighted in remark \ref{rem:BW}):
$$
\alpha = 0.2 \qquad \text{ and } \qquad \eta = -ik = -20i.
$$
Concerning the Brakhage-Werner integral equation and for simplicity, the parameter $\eta_{BW}$ is set to the optimal choice in the case of single scattering by a unit disk (see remark \ref{rem:BW} for references for other geometries):
$$
\eta_{BW} = i\frac{k}{2} = 10i.
$$

%%%%%%%%%%
%%%%%%%%% NEW
The four boundary integral equations EFIE \eqref{eq:EFIE}, MFIE \eqref{eq:MFIE}, CFIE \eqref{eq:CFIE} and Brakhage-Werner \eqref{eqEqInt:BW} are approximated using boundary element method of degree $1$, leading to four linear systems and four matrices denoted respectively by $\AEFIE$, $\AMFIE$, $\ACFIE$ and $\ABW$. Then, each linear system is preconditioned by its single scattering matrix, which is denoted by a hat symbol (\eg $\AEFIEp$ for $\AEFIE$). With these notation and according to theorems \ref{theo:equalityBIE} and \ref{theo:equalityBW}, the following equalities must hold true
$$
\AEFIEpm\AEFIE= \AMFIEpm\AMFIE = \ACFIEpm\ACFIE,
$$
and
$$
\AEFIEpm\AEFIE = (\AEFIE^{-1}\ABW)\ABWpm\ABW(\AEFIE^{-1}\ABW)^{-1}.
$$
It appears that, indeed, these four matrices are numerically close to each other:%, or at least with a satisfactory result:%, with a difference of
$$
\fracNinf{\AEFIEpm\AEFIE-\AMFIEpm\AMFIE}{\AEFIEpm\AEFIEpm} \leq 4.10^{-2}, \quad
\fracNinf{\AMFIEpm\AMFIE-\ACFIEpm\ACFIE}{\AMFIEpm\AMFIE} \leq 4.10^{-2},
$$
$$
\fracNinf{\AEFIEpm\AEFIE-\ACFIEpm\ACFIE}{\ACFIEpm\ACFIE} \leq  8.10^{-3},
$$
and
$$
\fracNinf{\AEFIEpm\AEFIE - (\AEFIE^{-1}\ABW)\ABWpm\ABW(\AEFIE^{-1}\ABW)^{-1}}{\ABWpm\ABW} \leq 4.10^{-4}.
$$
where $\|\cdot\|_{\infty}$ denotes the infinite matrix norm: $\|A\|_{\infty} = \max_{i} \sum_{j} |a_{i,j}|$ for any complex matrix $A = (a_{ij})$. 
%These simulations have been reproduced in other multiple scattering medium and for other frequencies with similar results.}
%%%%%%%%%%
Thus, these results illustrate theorems \ref{theo:equalityBIE} and \ref{theo:equalityBW}. From a practical point of view, these two propositions also imply that the matrices of the four preconditioned integral equations share the exact same spectrum. To observe this, Figures \ref{figEqInt:eigenvalues} and \ref{figEqInt:eigenvalues_zoom} present their eigenvalues in the complex plane. It appears that the four spectra coincide. More precisely, the relative error $\abs{a-b}/\abs{a}$ between two eigenvalues $a$ and $b$ reaches 2.7\% at its maximum. 
Hence, theorems \ref{theo:equalityBIE} and \ref{theo:equalityBW} seem to be numerically satisfied in the sense that the matrices of the preconditioned integral equations are very close to each other 
(up to a change of basis for the Brakhage-Werner integral equation). 
%(up to an invertible matrix for the Brakhage-Werner integral equation). 

To observe the effect of the preconditioner on an iterative solver,
%Figures \ref{figEqInt:eigenvalues} and \ref{figEqInt:eigenvalues_zoom} show a cluster of eigenvalues centered on the point $(1,0)$. To observe the effect 
the eight linear systems (four integral equations and four preconditioned integral equations) are solved using GMRES \cite{SaaSch86} with a restart of $50$ and a tolerance of $10^{-6}$. %, which is denoted by GMRES$(50, 10^{-6})$. 
The history of convergence, presented on Figures \ref{figEqInt:FEM_historique} and \ref{figEqInt:FEM_historique_zoom}, show that, first, EFIE and MFIE do no converge without being preconditioned. Adding to this, it appears that the preconditioner accelerates the convergence of the GMRES. Indeed, the number of iterations involved by the GMRES is decreased by $36\%$ and $23\%$ for respectively the CFIE and Brakhage-Werner. This can be a consequence of the cluster of eigenvalues centered on the point $(1,0)$, observed on Figures \ref{figEqInt:eigenvalues} and \ref{figEqInt:eigenvalues_zoom}. However, this paper is not intended to study the impact of the preconditioning on the convergence rate and the choice of the parameters is discussable. Nevertheless, these examples show that, first, the four curves representing the four preconditioned integral equations are superimposed, 
%which is natural since the four matrices share the same spectrum, 
and second that the preconditioner seems to accelerates and, at least, does not deteriorate the convergence rate. 

%Note that the EFIE and MFIE do no converge without being preconditioned. Now let us look at Figure \ref{figEqInt:FEM_historique_zoom}, which is a zoom on the four curves of interest, that is, the curve of the four preconditioned integral equations. Clearly, they are superposed to each other: the behavior of the GMRES is the same for the four preconditioned boundary integral equations. 

\begin{figure}[h!]
\centering
\subfigure[Example of a ``kite-shaped'' obstacle]{
\includegraphics[width=7.5cm]{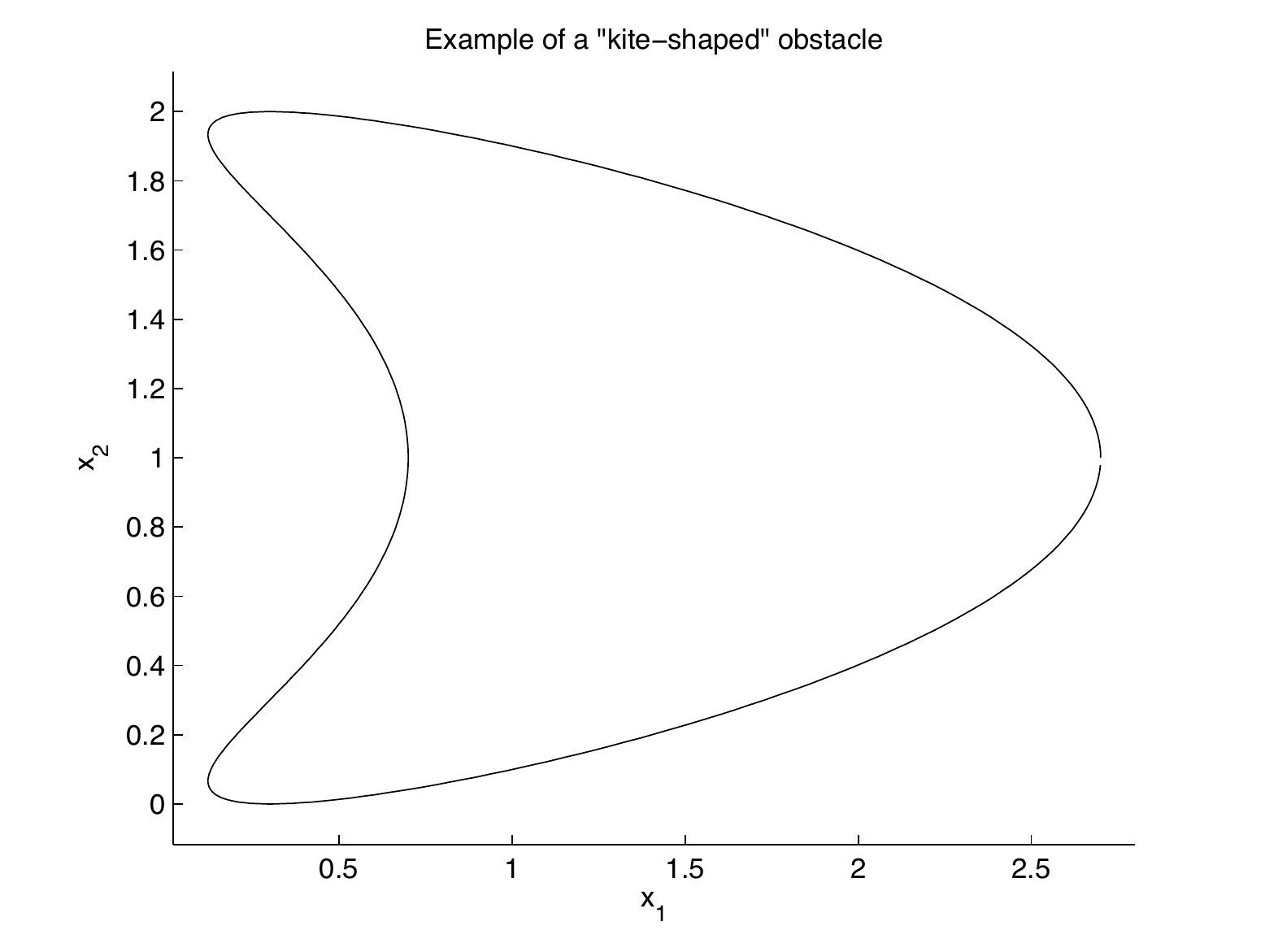}
\label{figEqInt:FEMKite}
}
\subfigure[Configuration]{
\includegraphics[width=7.5cm]{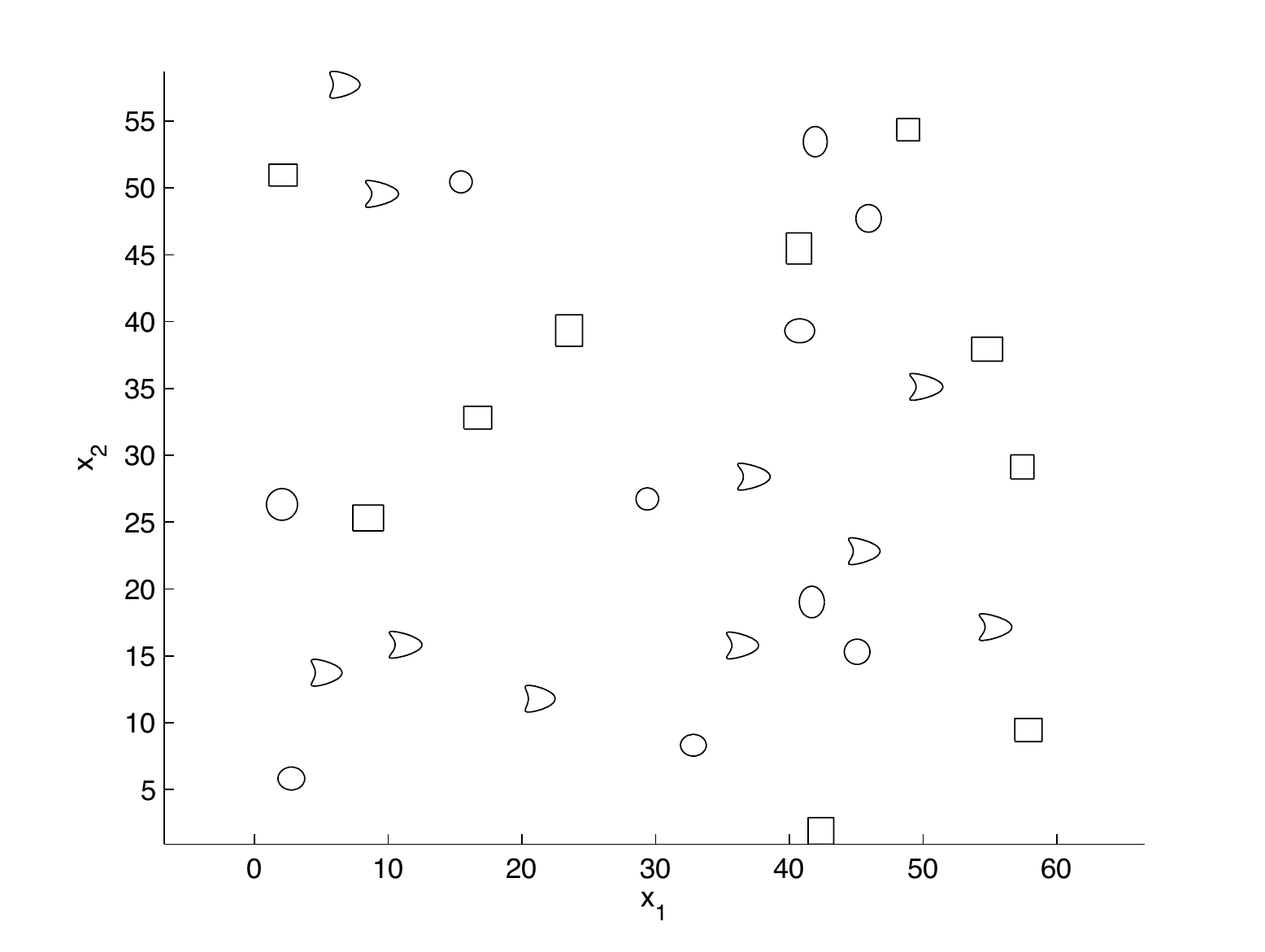}
\label{figEqInt:FEM_obstacles}
}
%\subfigure[Radar cross section]{
%\includegraphics[width=7.5cm]{RCS.pdf}
%\label{figEqInt:FEM_SER}
%}
\subfigure[Eigenvalues of the $4$ preconditioned operators]{
\includegraphics[width=7.5cm]{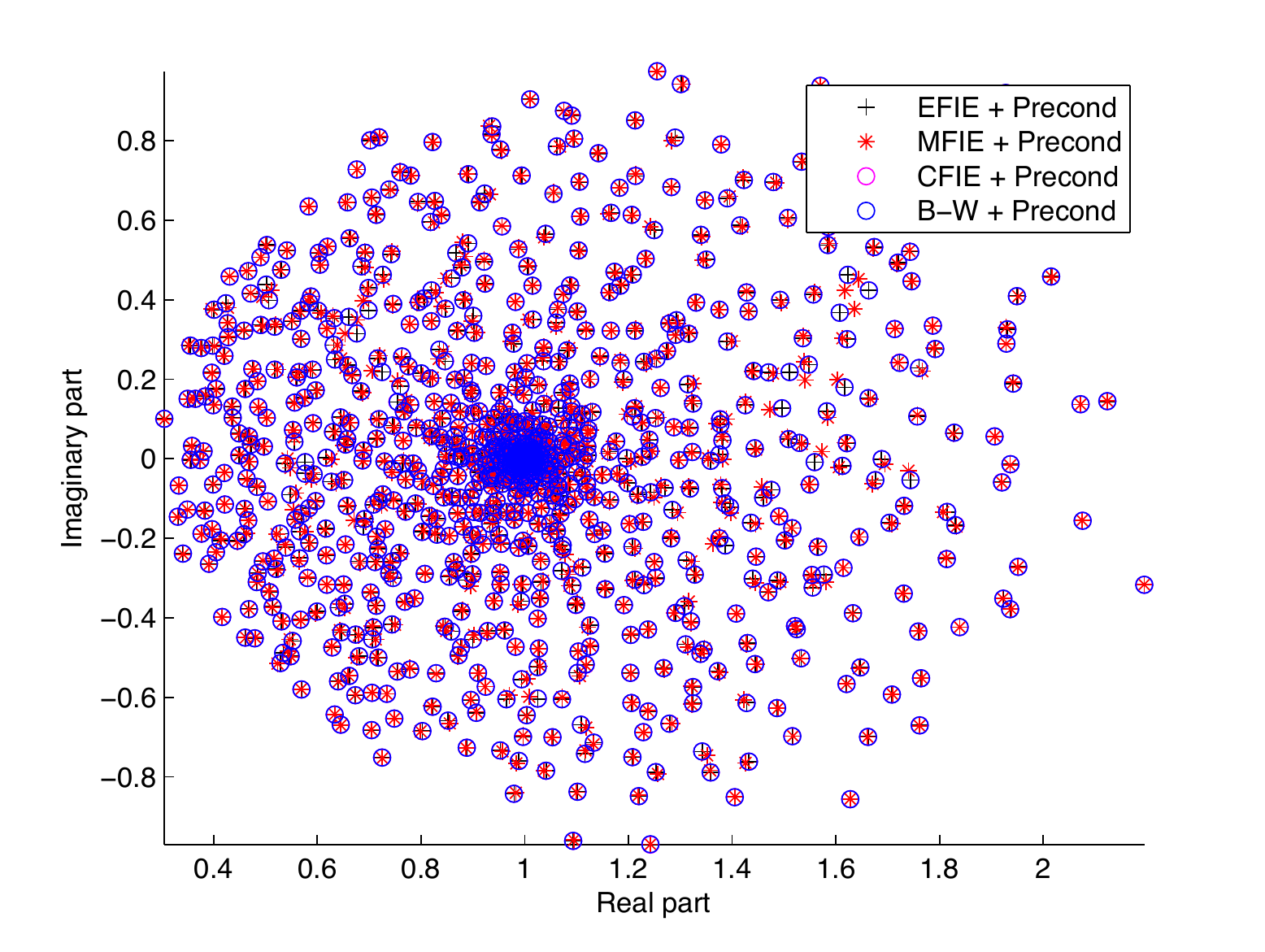}
\label{figEqInt:eigenvalues}
}
\subfigure[Zoom of Figure ({c}) around the point $(1,0)$]{
\includegraphics[width=7.5cm]{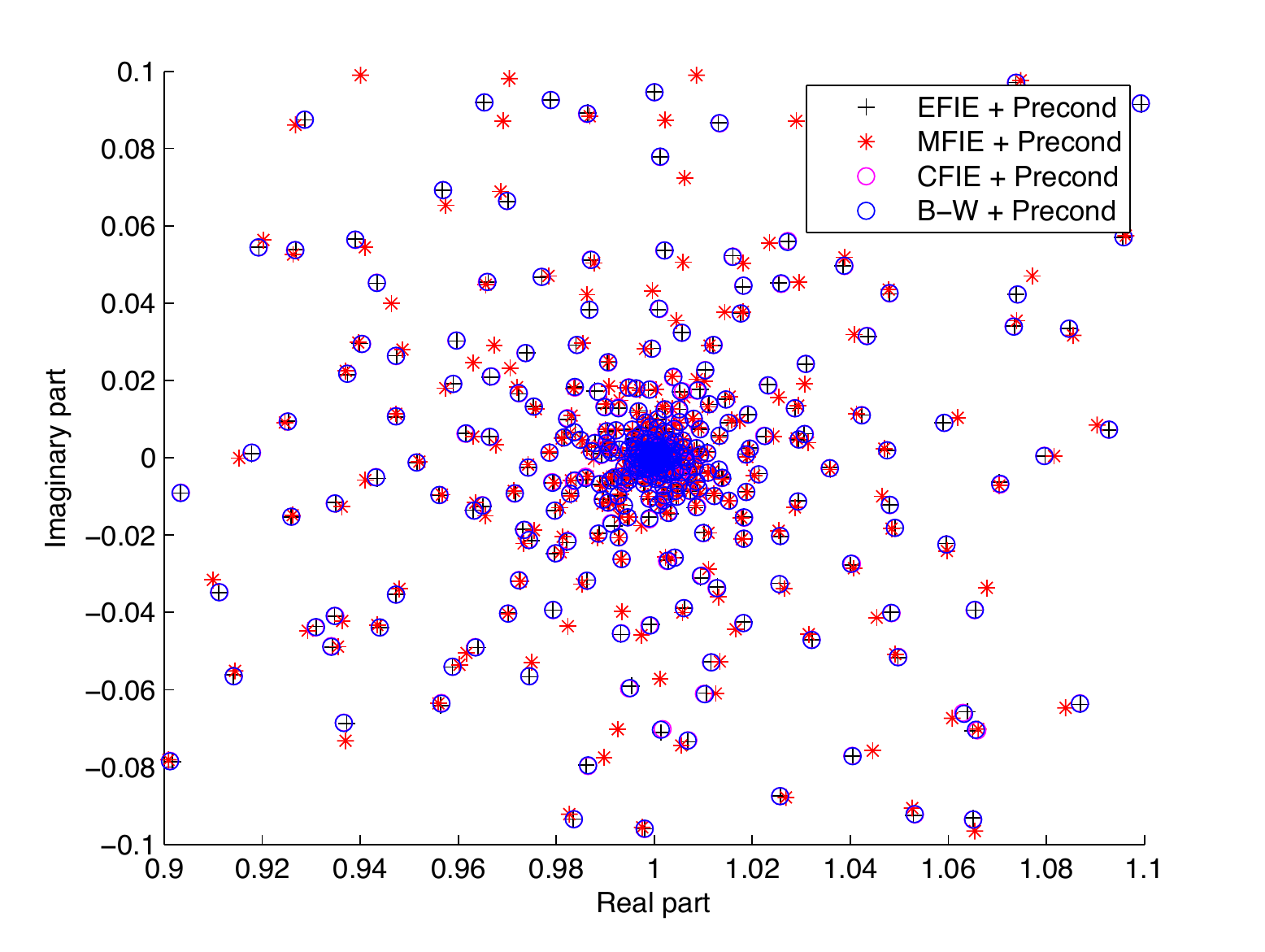}
\label{figEqInt:eigenvalues_zoom}
}
\subfigure[History of convergence of the GMRES($50$,$10^{-6}$)]{
\includegraphics[width=7.5cm]{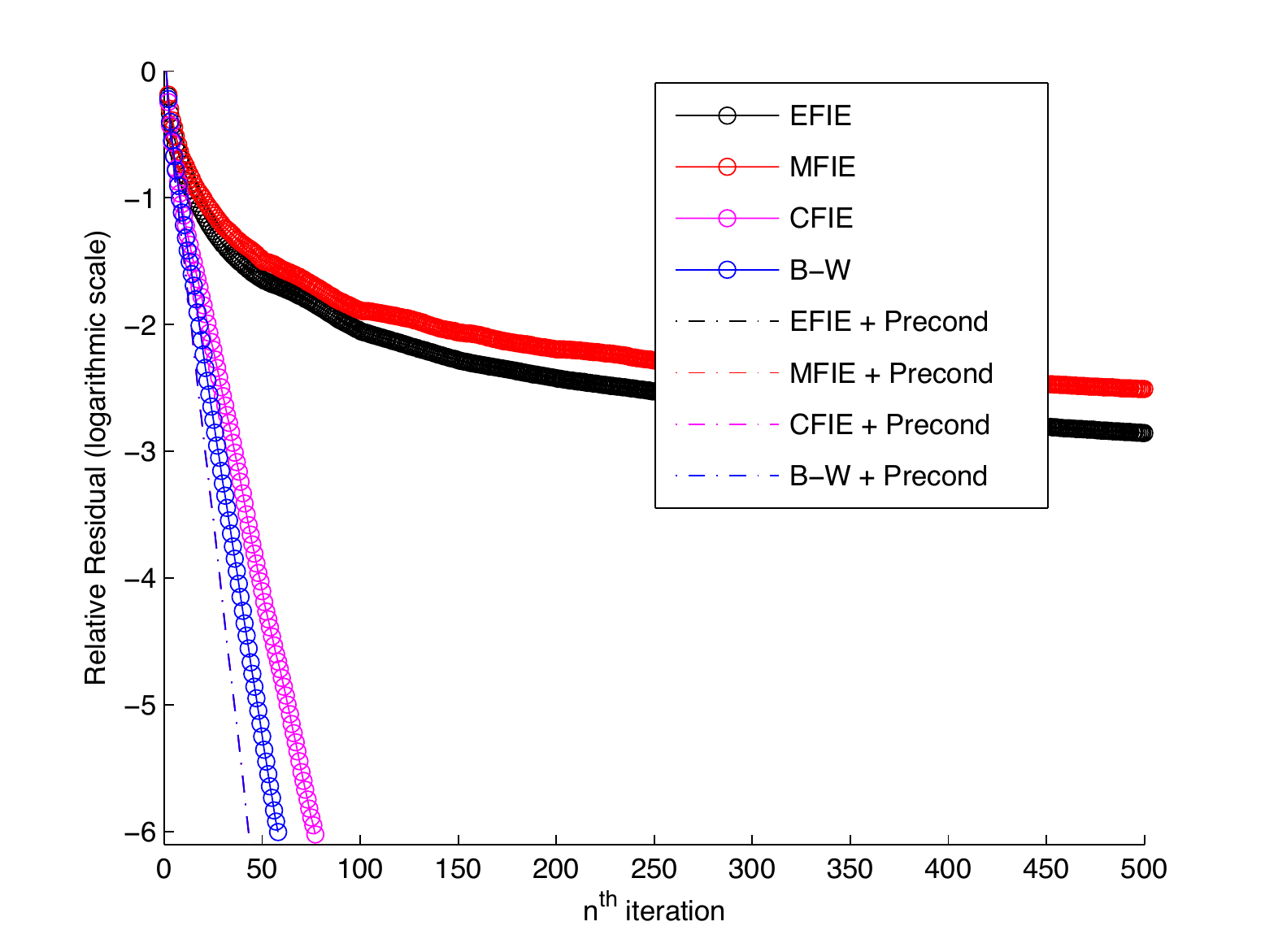}
\label{figEqInt:FEM_historique}
}
\subfigure[Zoom of Figure (e) on the four curves of interest]{
\includegraphics[width=7.5cm]{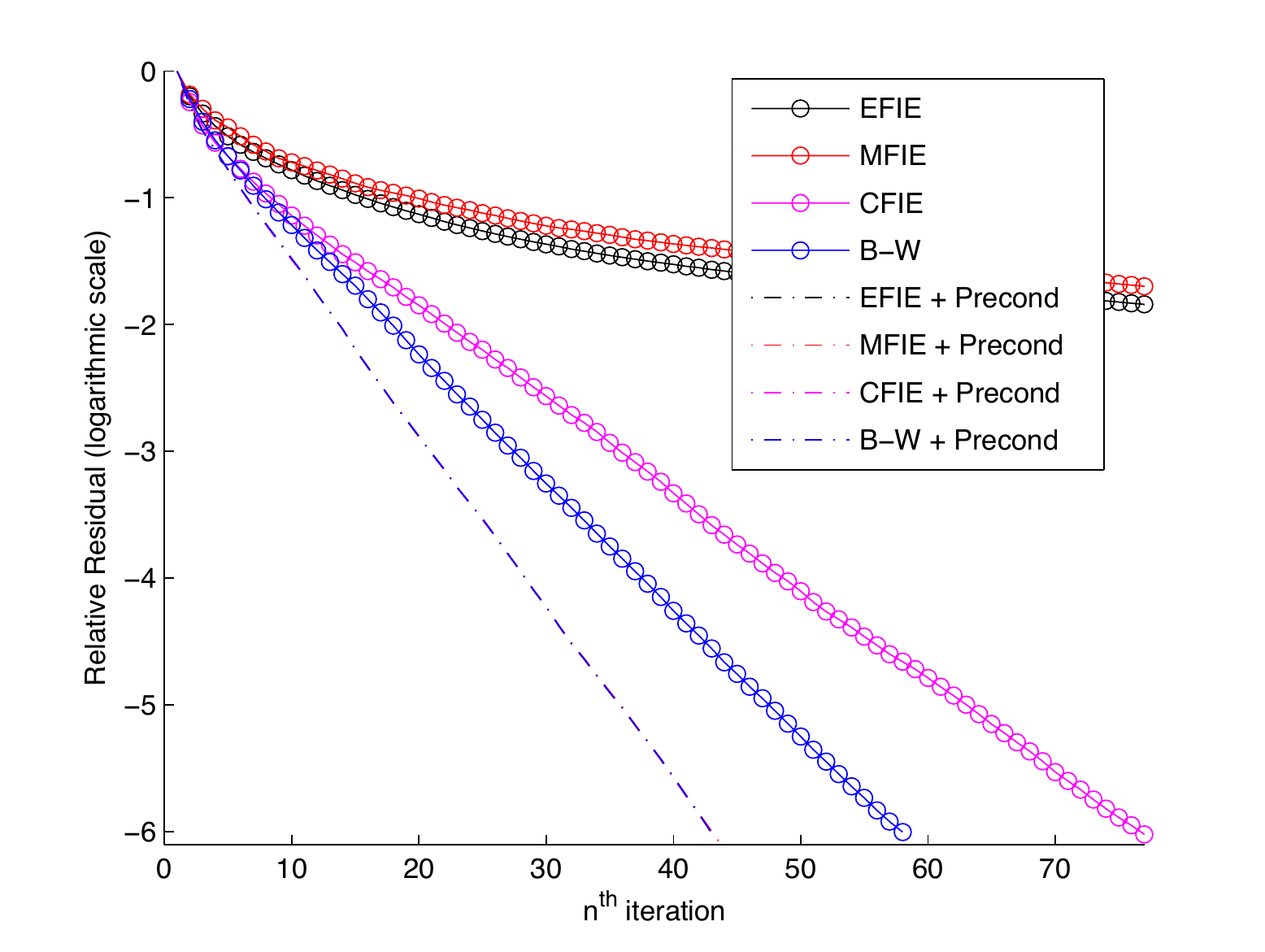}
\label{figEqInt:FEM_historique_zoom}
}
\caption{The ``kite-shaped'' obstacle is presented on figure (a) and picture (b) presents the $M=30$ obstacles with the three possible shapes: ellipses, rectangles and kytes. The scatterers are randomly placed in the box $[0,60]^{2}$, such that the distance $b_{pq}$ between the centers of two obstacles $\Omegamp$ and $\Omegamq$ satisfies $b_{pq}\geq 3$. figure ({c}) shows the (numerical) eigenvalues of the four preconditioned operators and figure (d) is a zoom around the point $(1,0)$.
%one can observe the radar cross section obtained for $k=20$ with the CFIE and the Brakhage-Werner integral equation (preconditioned or not). 
Finally, Figures (e) and (f) show the history of convergence of the GMRES($50$,$10^{-6}$) for the four integral equations and their preconditioned version (represented by ``+ Precond''). It appears that the four curves of the preconditioned integral equations are superimposed.}
\label{figEqInt:FEMConvergence}
\end{figure}

%\clearpage
%-----------------------------------------------------------------------------------------
\section{Conclusion}
%-----------------------------------------------------------------------------------------

This paper deals with boundary integral equation preconditioning for the multiple scattering problem. Two mains results were established, which can be summarized as follows. After being preconditioned by their single scattering operator, firstly, every direct integral equations lead to the exact same equation, and secondly, the indirect integral equation of Brakhage-Werner becomes 
similar to the direct integral equations.
%equal to the direct integral equations, up to an invertible operator. 
In particular, applying the single scattering preconditioner to whichever integral formulation leads to the exact same convergence rate of the Krylov subspaces solver.

To conclude this article, two brief remarks can be done. First, the above results could probably be extended to the Maxwell's equations. Second, it should be kept in mind that the single scattering preconditioner involves the inversion of the $M$ diagonal blocks of the boundary integral operator. Numerically, these last operations strongly depend on the considered integral formulation and moreover, at high frequency, they become costly and must be handled by iterative methods. On the other hand, the multiple scattering problem is also treated by an ``outer'' Krylov subspace solver. Hence, this will lead to launch $M$ inner Krylov solvers at each iteration of the outer solver.

\section*{Acknowledgement}

The author would like to express his sincere gratitude to the referee, X. Antoine and K. Ramdani for their helpful comments and suggestions.

%-----------------------------------------------------------------------------------------

\bibliography{Bibliographie}
%\bibliography{/Users/bertrand/Documents/Bibliographie/Bibliographie}
%\addcontentsline{toc}{chapter}{Bibliographie}
%\bibliographystyle{alpha}
\bibliographystyle{plain}

\end{document}